\documentclass[12pt]{article}
\usepackage{amsfonts,amssymb,amscd,color}
\usepackage{oldgerm}

\usepackage{hyperref}
\definecolor{Red}{cmyk}{0,1,1,0}

\definecolor{verde}{cmyk}{1,0,1,0}

\definecolor{azul}{cmyk}{1,1,0,0}

\usepackage[Conny]{fncychap}
\usepackage{amssymb,amsmath,amsthm,amsfonts}
\usepackage[latin1]{inputenc}
\usepackage{mathrsfs}
\usepackage[all]{xy}
\newtheorem{teo}{Theorem}[section]
\newtheorem{lema}[teo]{Lemma}
\newtheorem{defnc}[teo]{Definition}
\newtheorem{corollary}[teo]{Corollary}
\newtheorem{prop}[teo]{Proposition}
\newtheorem{obs}[teo]{Remark}

\newcommand{\ba}{\textbf{a}}

\newcommand{\rar}{\rightarrow}
\begin{document}

\title{
Spectral Properties of the Ruelle Operator 
on the Walters Class over Compact Spaces}

\author{Leandro Cioletti\qquad and\qquad Eduardo A. Silva}
\date{\today}
\maketitle

\begin{abstract}
Recently the Ruelle-Perron-Fr\"obenius theorem was proved for H\"older potentials 
defined on the symbolic space $\Omega=M^{\mathbb{N}}$, 
where (the alphabet) $M$ is any compact metric space.
In this paper, we extend this theorem to the Walters space $W(\Omega)$, 
in similar general alphabets.
We also describe in detail an abstract procedure 
to obtain the Fr\'echet-analyticity of the Ruelle
operator under quite general conditions and 
we apply this result to prove the analytic dependence 
of this operator on both Walters and H\"older spaces.
The analyticity of the pressure functional on H\"older spaces
is established. An exponential decay of the
correlations is shown when the Ruelle operator has the spectral gap property. 

A new (and natural) family of Walters potentials 
(on a finite alphabet derived from the Ising model)
not having an exponential decay of the correlations is presented. 
Because of the lack of exponential decay, for such potentials we have the 
absence of the spectral gap for the Ruelle operator.
The key idea to prove the lack of exponential decay  
of the correlations are the Griffiths-Kelly-Sherman inequalities.

\end{abstract}
\noindent 
{\small {\bf Key-words}: 
Thermodynamic formalism, Ruelle operator,  
one-dimensional lattice, Analyticity of Pressure,
Spectral Gap.}
\\
\noindent 
{\small {\bf MSC2010}: 37A60, 37A50, 82B05.}
\section{Introduction}

The Ruelle-Perron-Fr\"obenius theorem is one of the most 
important results in modern thermodynamic formalism.
Nowadays the Ruelle operator have become a standard tool in many 
different areas in dynamical systems and
other branches of Mathematics and Mathematical Physics.
The literature about Ruelle-Perron-Fr\"obenius theorem is vast, 
the following is a partial list of books and papers
on this subject \cite{Baladi,bo,Fan,FJ,Jiang,PP,Ruelle-1968,Viana}.

The classical thermodynamic formalism was in its origin developed 
in the Bernoulli space $M^{\mathbb{N}}$, with $M$ being a finite alphabet, see \cite{PP}. 
This assumption on $M$ allows one conjugates, using a Markov partitions, 
the shift maps on the Bernoulli space with uniform hyperbolic maps in differentiable manifolds.
 
The Ruelle operator formalism is also proved useful in the 
multifractal analysis context. Bowen in the seminal work
\cite{boo} established a relationship between Hausdorff dimension 
of certain fractal sets and topological pressure in the 
context of conformal dynamics in one dimension. This technique
is known as Bowen's equation  and in subsequent works it
was extended by Manning and McCluskey for dynamical systems on surfaces 
aiming to compute fractal dimension of horseshoes, see for more
details \cite{boo,ma,MM} 
and also the introductory texts \cite{Barreira,Pesin}.

The motivation to consider more general alphabets from 
dynamical system point of view is given, 
for example, in \cite{Sarig 1,Sarig 2} 
where proposed models with infinite alphabet $M=\mathbb{N}$ is used to 
describe non-uniformly hyperbolic maps, 
for example, the Manneville-Pomeau maps.
In Classical Statistical Mechanics uncountable alphabets 
shows up, for example, in the so-called $O(n)$ models
with $n\geq 2$. In these models the alphabet is
$\mathbb{S}^{n-1}$ the unit sphere in the 
Euclidean space $\mathbb{R}^n$, for details see \cite{GJ-Book}.
Unbounded alphabets as general standard Borel spaces, which includes 
compact and non-compact, are considered in details in \cite{Geogii88}.
We should mention that ergodic optimization problems are also being considered 
in infinite countable alphabets, see \cite{BG,Daon,JMUrb,Sarig 1}.

Walters work \cite{PW78} marked the beginning of three 
decades of great activity in thermodynamic formalism 
where less regular than H\"older potentials were considered. 
This class of potentials is called Walters class or alternatively
Walters space.
A rather complete theory in Walters space were developed
for finite and countable alphabets, see \cite{Daon}. 
In the Walters paper, the dynamical system is supposed to be defined on a class of
compact sets, expansive and mixing and 
the potentials can be any positive summable 
variation function.

The aim of this work is to extend some results obtained in \cite{PW78}
as the Ruelle-Perron-Fr\"obenius theorem, 
the analytic dependence the Ruelle operator with respect
to the potential, spectral properties of this operator and some consequences
of either presence or absence of the spectral gap, as 
for example the pressure analyticity.
The main difficulty in carrying out the construction of the Ruelle operator
for uncountable alphabets is overcome by the introduction of 
what we call an {\it a-priori} probability measure on $M$,
which is a common strategy in the theory of general DLR-Gibbs measures.

This paper is organized as follows. In Section 2, we prove under 
quite general conditions the analyticity of the Ruelle operator and 
its dual with 
respect to the potential $f$. 
The result is more general, in fact we show that if the Ruelle operator 
$\mathscr{L}_f$ leaves invariant a Banach algebra $\mathcal{K}\subset C(\Omega)$
for all $f\in \mathcal{K}$, then the maps 
$\mathcal{K} \ni f\to \mathscr{L}_f$, $\mathcal{K}^{*} \ni f\to \mathscr{L}_f^{*}$ 
are analytic. 
We use this result in Section 5 with $\mathcal{K}=C^{\gamma}(\Omega)$ 
to obtain the analyticity of the topological pressure 
$P: C^{\gamma }(\Omega)\to \mathbb{R}$, which extends the analogous results 
known in finite/discrete alphabets. 
We shall remark that even for discrete case, despite this being a folkloric
result we were not able to find its rigorous proof in the literature.

Section 3 deals with the Walters space, notation $W(\Omega)$. 
In the discrete setting there are 
several known equivalent characterizations of the Walters condition,
however in a more general setting as compact metric spaces, things 
are more subtle. We show that the natural generalization 
of the two most popular characterizations of the Walters class are not 
equivalent when the alphabet is uncountable. We give an explicit example illustrating 
this fact. We introduce what we call weak and strong Walters conditions.
Finally a generalization of \cite{PW78} is proved for general compact
alphabets. We remark that this theorem is also a non-trivial generalization
of one the main theorems in \cite{le,LMMS} where the Ruelle operator
on uncountable alphabet are taken into account.

In Section 7 we introduce a new family of potentials for which 
the Ruelle operator has absence of the spectral gap. This section
is heavily based on the ideas borrowed from Statistical Mechanics
and the Griffiths-Kelly-Sherman (GKS) inequalities, 
\cite{griffiths1,griffiths2,KS68,ginibre70}. 
For the convenience of the reader we precisely stated all the theorems
we need only in the needed generality but 
we provide its classical references where general settings
are presented. Some Ising model routine computations are 
presented in details in order to make 
our exposition self-contained
for non-specialists in Statistical Mechanics.
These potentials belong to a infinite dimensional linear
subspace of $C(\Omega,\mathbb{R})$ whose intersection
with the Walters space is an infinite dimensional linear 
subspace not contained in the H\"older space. On that
space Dobrushin in \cite{dobrushin73} 
by using estimates of mean value of 
exponential functionals of random processes 
and latter Cassandro and Olivieri 
\cite{CO81}
 employing a renormalization 
group idea together with the cluster expansion 
proved analyticity of the pressure. 
It worth mention that the subexponential 
decay obtained in Section 5
can not be recovered from the seminal Sarig's work \cite{Sarig02}
about subexponential decay of correlations
and neither from the improvement provided by 
Gou\"ezel in \cite{Gouezel04}. 
The examples presented in Section 5 can 
shed some light in potential applications of the 
GKS inequalities to study the absence of the spectral gap
in other situations.

\section{Basic Definitions}

In this section we setup the notation and present 
some preliminaries results.
Let $M=(M,d)$ be a compact metric space, equipped 
with a Borel probability measure $\mu:\mathscr{B}(M)\to [0,1]$ 
having the whole space $M$ as its support.  
We shall denote by $\Omega=M^{\mathbb{N}}$ 
the set of all sequences $x=(x_1,x_2,\ldots )$, where 
$x_i \in M$, for all $i\in {\mathbb N}$. 
We denote by $\sigma:\Omega\to\Omega$ 
the left shift mapping which is given 
by $\sigma(x_1,x_2,\ldots)=(x_2,x_3,\ldots)$.
We consider the metric $d_{\Omega}$ on $\Omega$ given by
\[
d_{\Omega} (x,y)
=
\sum_{n=1}^{\infty} \frac{1}{2^n}d(x_n,y_n)
\]
The metric $d_{\Omega}$ induces the product topology
and therefore follows from the Tychnoff Theorem that 
$(\Omega,d_{\Omega})$ is a compact metric space.
The space of all the continuous real functions $C(\Omega, \mathbb R)$
is denoted simply by $C(\Omega)$. 
For any fixed $0\leq \gamma\leq 1$ we denote by $C^{\gamma}(\Omega)$ 
the space of all $\gamma$-H\"older continuous functions, 
i.e, the set of all functions $f:\Omega\to\mathbb{R}$ satisfying 
\[
\mathrm{Hol}(f)
=
\sup_{x,y\in\Omega: x\neq y}
\dfrac{|f(x)-f(y)|}{d_{\Omega}(x,y)^{\gamma}}
<+\infty.
\] 

We equip $C^{\gamma}(\Omega),~0\leq \gamma\leq 1$ with the norm 
$\|\cdot\|_{\gamma}$ which is defined for $\gamma=0$ by
$\|f\|_0=\sup_{x\in \Omega} |f(x)|$ and for 
$0< \gamma\leq 1$ by  
$\|f\|_{\gamma}= \|f\|_0+ \mathrm{Hol}(f)$.
We recall that 
$(C^{\gamma}(\Omega),\|\cdot\|_{\gamma}) $ is a Banach space for
any $0\leq \gamma \leq 1$.

Our {\bf potentials} will be elements of $C(\Omega)$
and in order to have a well defined Ruelle operator 
when $(M,d)$ is a general compact metric space we need 
to consider an {\it a priori measure} which is simply a 
Borel probability measure $\mu:\mathscr{B}(M)\to \mathbb{R}$,
where $\mathscr{B}(M)$ denotes the Borel $\sigma$-algebra 
of $M$. For many of the most popular choices of an uncountable
space $M$ there is a natural {\it a priori measure} $\mu$. 
Throughout this paper the a priori measure $\mu$ 
is supposed to have the whole space $M$ as its support. 
The Ruelle operator 
$\mathscr{L}_f: C^{\gamma}(\Omega) \to C^{\gamma}(\Omega)$
is the mapping sending $\varphi \mapsto \mathscr{L}_{f}(\varphi)$ 
which is defined for any $x\in\Omega$ by the expression
\[
\mathscr{L}_f(\varphi)(x)
=
\int_M e^{f(ax)}\varphi(ax)d\mu(a),
\]
where $ax$ denotes the sequence 
$ax=(a, x_1, x_2, \ldots)\in \Omega$. 

This operator is a generalization of the classical 
Ruelle operator and has been appeared lately 
in the Thermodynamic Formalism literature, see for example 
\cite{le,LMMS,ERR}.
The classical Ruelle operator can be recovered on this  
setting by considering $M=\{0,1,\ldots,n\}$ and 
the {\it a priori} $\mu$ as the counting measure. 
Our starting point is the following theorem.

\begin{teo}[Ruelle-Perron-Fr\"obenius]\label{Ruelle-Perron}
Let $(M,d)$ be a compact metric space, $\mu$ a 
Borel probability  measure of full support on $M$ and 
$f$ be a potential in $C^{\gamma}(\Omega)$, 
where $0<\gamma<1$. Then
$\mathscr{L}_f: C^{\gamma}(\Omega) \to C^{\gamma}(\Omega)$
have a simple positive eigenvalue of maximal modulus $\lambda_f$, 
and there are a strictly positive function $h_f$ 
and a Borel probability measure $\nu_{f}$ on $\Omega$ such that,
\begin{itemize}
\item[i)] 
the remainder of the spectrum of 
$\mathscr{L}_f: C^{\gamma}(\Omega) \to C^{\gamma}(\Omega)$
is contained in 
a disc of radius strictly smaller then $\lambda_f$;

\item[ii)] 
for all continuous functions $\varphi\in C(\Omega)$ we have
\[
\lim_{n\to\infty}
\left\|
\lambda_{f}^{-n}\mathscr{L}^{n}_{f}\varphi-h_f\int_{\Omega} \varphi\,  d\nu_{f}
\right\|_0
=
0.
\]
\end{itemize}
\end{teo}

\begin{proof}
See \cite{le} for the case $M=S^1$ and \cite{LMMS} 
for more general compact metric spaces.
\end{proof}

\begin{obs}
Strictly speaking, in \cite{le} and \cite{LMMS} the item $(ii)$ above was proved
only for Hölder continuous potentials. 
However this is enough since the space of the Hölder continuous potentials 
is dense in $(C(\Omega),\|\cdot\|_{0})$.  
Therefore a straightforward computation shows that 
the convergence on the item $(ii)$ holds for all $\varphi\in C(\Omega)$.
The denseness of $C^{\gamma}(\Omega)$, $0<\gamma<1$ in $C(\Omega)$ 
is a consequence of the Stone Weierstrass Theorem. 
Indeed, $C^{\gamma}(\Omega)$ is an algebra of functions 
containing all the constant functions and if $x\neq y \in \Omega,$ 
then the function $f$ given by $f(y)=d_{\Omega}(y,x)^{\gamma}$, 
separates $x$ and $y$ and $f\in C^{\gamma}(\Omega)$. 
Since $\Omega$ is compact the result follows.
 
\end{obs}

Following \cite{le,LMMS} we define 
the entropy of a shift invariant measure $\nu\in \mathcal{M}_{\sigma}$ 
and the pressure of the potential $f$, respectively, as follows
\[
h(\nu)
=
\inf_{f \in C^{\gamma}(\Omega)}
\left\{-\int_{\Omega}f d\nu+ \log\lambda_{f}  \right\}
\qquad \text{and}\qquad
P(f)
=
\sup_{\nu\in \mathcal{M}_{\sigma}}
\left\{h(\nu)+\int_{\Omega}f\, d\nu \right \}.
\]

\begin{prop}\label{Principio variacional}
For each $f \in C^{\gamma}(\Omega)$ we have for all 
$x\in \Omega$ that 
\[
P(f)
=
\lim_{n\to\infty}
\frac{1}{n}
\log[ \mathscr{L}_{f}^n(1)(x)]
=
\log\lambda_{f}.
\]
\end{prop}

\begin{proof}
See \cite{LMMS} Corollary 1.
\end{proof}

Another important property of the Ruelle operator is its 
analytic dependence (in the Fr\'echet's sense) 
with respect to the potential. 
The lemma below state it precisely.

\begin{lema}\label{Lema principal}
The map
$
\Theta :C^{\gamma}(\Omega)
\rightarrow 
L(C^{\gamma}(\Omega),C^{\gamma}(\Omega))
$ 
sending $f \in C^{\gamma}(\Omega)$ to the
Ruelle operator ${\mathscr{L}}_{f}$ associated to 
the potential $f$, is an analytic map.
\end{lema}

\begin{proof}
See \cite{ERR}  Theorem 3.5.
\end{proof}

One of the aims of this work is to extend the previous mentioned results to potentials 
on the {\bf Walters space} $W(\Omega)$ (to be defined in the next section), 
where $\Omega$ is the infinite Cartesian product of 
a general compact metric space $(M,d)$.

We shall prove  that the Ruelle operator
and its dual depends analytically on the potential 
in $C^{\gamma}(\Omega)$ and then derive the analyticity of the pressure. 
In order to formulate these results using an unified setting, 
we need to introduce some additional notation. 
Let $\mathcal{K}\subset C(\Omega)$ be an arbitrary linear subspace 
of $C(\Omega)$, endowed with a norm $\|\cdot\|$. 
We use the notation $\mathcal{K}^{*}$ to denote 
the topological dual of $(\mathcal{K},\|\cdot\|)$.
As usual, we define the norm 
of an element $\phi \in \mathcal{K}^{*}$ by
$
\|\phi\|_{*}=\sup\{|\phi(f)|: f\in \mathcal{K}\ \text{and}\  \|f\|=1\}.
$
To lighten the notation, the space $L(\mathcal{K},\mathcal{K})$ 
of the all continuous (strong topology) linear 
operators acting on $\mathcal{K}$ will be denote 
by $V\equiv L(\mathcal{K},\mathcal{K})$.

\begin{defnc}
Let $\mathcal{K}\subset C(\Omega)$ be a linear subspace. 
We say that $\mathcal{K}$ is invariant  for the Ruelle operator, 
if for all $f\in \mathcal{K}$ 
we have $\mathscr{L}_f\mathcal{K}\subset \mathcal{K}.$
\end{defnc}
The central examples of invariant subspaces 
for the  Ruelle operator appearing 
here are the spaces $C^{\gamma}(\Omega)$, 
$0<\gamma\leq 1$, 
and the {\it Walters Space} $W(\Omega)$.

The next proposition plays a key role in the study of 
the analyticity of the pressure. 
Its first statement is a simple generalization 
of the Theorem $3.5$ in \cite{ERR}, 
which is presented here for the reader's convenience.
The second one follows from the first after some work.

\begin{prop}\label{An.Op.Ru}
Suppose that $\mathcal{K}\subset C(\Omega)$ is equipped 
with a norm $\|\cdot \|$ so that $(K, \|\cdot\|)$ 
is a Banach Algebra, $\mathcal{K}$ is invariant for the Ruelle operator
and for any $f\in\mathcal{K}$ assume that $\mathscr{L}_f \in V$. 
Then both mappings $\Theta$ and $\Theta^{*}$ given by 
$$
\mathcal{K}\ni f\mapsto \mathscr{L}_f\in L(\mathcal{K},\mathcal{K})
\qquad\text{and}\qquad
\mathcal{K}\ni f\mapsto \mathscr{L}^{*}_f\in L(\mathcal{K}^{*}, \mathcal{K}^{*})
$$ 
define analytic functions.

\end{prop}

Before proving the above proposition
we state an immediate corollary of it 
which is an important tool to obtain the 
analyticity of the pressure functional.

\begin{corollary}
For each $0<\gamma\leq 1$, both mappings  
$$
C^{\gamma}(\Omega)\ni 
f\mapsto 
\mathscr{L}_f\in 
L(C^{\gamma}(\Omega),C^{\gamma}(\Omega))
\qquad\text{and}\qquad
C^{\gamma}(\Omega)\ni 
f\mapsto 
\mathscr{L}^{*}_f\in 
L(C^{\gamma}(\Omega)^{*}, C^{\gamma}(\Omega)^{*})
$$ 
define analytic maps.
\end{corollary}
\begin{proof}
Notice that the subspace $C^{\gamma}(\Omega)$ 
is invariant for the Ruelle operator and 
$(C^{\gamma}(\Omega),  \|\cdot\|_{\gamma})$
is a Banach Algebra (see \cite{ERR}), so we are done.
\end{proof}

\noindent {\bf Proof of the Proposition \ref{An.Op.Ru}.}
\\
We first prove the \noindent analyticity of $\Theta$.  
Given $f,h\in \mathcal{K}$ and $\varphi\in C(\Omega)$ 
we have for any $x\in \Omega$ the following equality
\begin{align*}
\Theta (f + h)(\varphi)(x) - \Theta (f)(\varphi)(x)
=& \;\; {\mathscr{L}}_{f +h}(\varphi)(x) -{\mathscr{L}}_{f}(\varphi)(x)  
\\[0.3cm]
=& 
\int \limits_{M} e^{f (ax) +h (ax)} \varphi (ax) d\mu (a)
- 
\int \limits_{M} e^{f (ax) } \varphi (ax) d\mu (a) \\
=& 
\int \limits_M e^{f (ax) } \varphi (ax) \big( e^{h (ax)} -1 \big)d\mu (a) \\
=& 
\int \limits_M \biggl(e^{f (ax) } 
\varphi (ax) \sum \limits_{n=1} \limits^{\infty} 
\frac{[h (ax)]^n}{n!} \biggr)d\mu (a)\,.
\end{align*}
As long as the Fubini Theorem applies we get 
$$
\Theta (f + h)(\varphi)(x) - \Theta (f)(h)(x) 
=
\sum \limits_{n=1} 
\limits^{\infty}\frac{1}{n!}
\int \limits_M e^{f (ax)}\varphi (ax)[h (ax)]^n d\mu (a),
$$
which is equivalent to 
$
\Theta (f + h)(\varphi)(x) - \Theta (f)(\varphi)(x) 
= 
\sum_{n=1}^{\infty}\frac{1}{n!}
\Theta (f)(\varphi\cdot h^n)(x).
$
This equality can be rewritten, by omitting the dependence on $x$ and $\varphi$,
simply as follows
\begin{equation}
\label{representacaoemseriesdepotencias}
\Theta (f + h) - \Theta (f) = \sum \limits_{n=1} 
\limits^{\infty}\frac{1}{n!}
\Theta (f)((\cdot) h^n)\,.
\end{equation}
To justify the applicability of the Fubini Theorem in 
this case is sufficient to prove that the above sum converges
in $(\mathcal{K},  \|\cdot\|)$. 
We first observe that for any $h _1 ,\ldots ,h _k$ in $\mathcal{K}$, 
the mapping
$
\varphi \mapsto \Theta (f )(\varphi h _1 \ldots h_k )
$
from $\mathcal{K}$ to itself defines 
a continuous linear operator, i.e.,
$\Theta (f )((\cdot) h _1 \ldots h_k ) \in  V$
whose its norm is bounded by
\begin{eqnarray}
\label{dominando}
\| \Theta (f )((\cdot) h _1 \ldots h_k ) \|_{V} 
\leq 
\| {\mathscr L}_{f} \|_{V}  \|h _1\|\ldots \|h_k \|.
\end{eqnarray}
which is proved by a routine computation using that $\mathscr{L}_{f}\in V$ and 
$(\mathcal{K}, \|\cdot \|)$ is a Banach algebra.
%
As a consequence of this inequality we get 
$$
\sum \limits_{n=1} 
\limits^{\infty}\frac{1}{n!} 
\| \Theta (f)((\cdot) [h]^n) \|_{V}
\leq 
\sum \limits_{n=1} \limits^{\infty}
\frac{1}{n!}
\|\mathscr L _{f }\|_{V} (\| h \|)^n 
= 
\| \mathscr L _{f }\|_{V} 
\left\{ e^{ \| h \|} -1 \right\}
$$
which immediately implies that the series 
$
\sum_{n=1}^{\infty}(1/n!)
\Theta (f)((\cdot) [h]^n)
$ 
converges in $ V$.

\noindent {\bf Claim 1.} 
For any $k\in\mathbb{N}$ and 
$h _1 ,\ldots, h_k \in \mathcal{K}$, 
we have that 
\[
D^k \Theta (f )(h _1, \ldots , h _k)
= 
\Theta (f ) ((\cdot)h _1 ,\ldots , h _k).
\]
The verification will be carried out by induction on $k$.
In what follows  $L^k=L^k(\mathcal{K}, V)$ 
denotes the set of all continuous $k-$linear  
functions 
$l:\mathcal{K} \times \ldots \times \mathcal{K}\rightarrow  V$, 
from $\mathcal{K} \times \ldots \times \mathcal{K}$ 
($k-$ copies of   $\mathcal{K} $) into $ V$. 
The norm $\|\cdot\| _{L^k}$ of $L^k$ is given by
$$
\| l \| _{L^k}
=
\sup \limits_{\stackrel{\|h _i\|=1}{i=1,\ldots ,k}}
\| l(h_1,\ldots,h_k)\| _{V}
\;,\;\; l\in L^k\,.
$$
Let us prove that the statement is true for $k=1$: in fact, 
by using (\ref{representacaoemseriesdepotencias}) we have,
$$
\Theta (f + h _1) - \Theta (f ) 
= 
\Theta (f )( (\cdot)h _1) + \mathcal{O}_1(h _1)
$$
where, 
$
\mathcal{O}_1(h_1 )
=
\sum_{n=2}^{\infty} (1/n!)\Theta (f )((\cdot)[h_1]^n)$. 
The inequality~(\ref{dominando}) implies that
$
\| \Theta (f )((\cdot)h _1) \|_V \leq \| \mathscr L _{f} \|_{V} \| h_1 \|
$
and thus the mapping 
$h _1 \mapsto \Theta (f )((\cdot)h _1)$ is in $L^1$. 
Again, in view of the inequality~(\ref{dominando}) we have
\begin{align*}
\| \mathcal{O}_1 (h _1) \| _{V}
=  
\| \sum \limits_{n=2}\limits^{\infty} 
	\frac{1}{n!}\Theta (f )((\cdot)[h_1]^n) 
\| _{V}  
\leq  
\sum \limits_{n=2}\limits^{\infty} 
\frac{1}{n!}
\| \mathscr L _{f} \| _{V} 
(\| h \| )^n  
\end{align*}
showing that  
$ (1/\| h _1 \|) \| \mathcal{O}_1 (h _1)\| _{V}\rightarrow 0$, 
when $\| h _1 \|  \rightarrow 0$. 
Therefore, 
$D^1 \Theta (f)(h _1)=\Theta (f )((\cdot)h _1)$
and the statement is true for $k=1$.
Now, let us suppose the statement is true for $k-1,\;k\geq 2$, i.e.,
\begin{eqnarray}
\label{k-1-derivada}
D^{k-1} \Theta (f )(h _1, \ldots , h _{k-1})
= 
\Theta (f ) ((\cdot)h _1 \ldots h _{k-1})
\;,\;\; h _1 ,\ldots, h_{k-1}\in \mathcal{K}\,.
\end{eqnarray}
We shall verify that the statement is true for $k$, i.e.
\begin{eqnarray}
\label{kderivada}
D^k \Theta (f )(h _1, \ldots , h _k)
= 
\Theta (f ) ((\cdot)h _1 .\ldots . h _k)
\;,\;\; h _1 ,\ldots, h_k \in \mathcal{K}\,.
\end{eqnarray}
By the induction hypothesis \eqref{k-1-derivada}, 
given $h _1, \ldots h _{k-1}, h _k$ and $h $ in $\mathcal{K}$ we have,
\begin{multline*}
D^{k-1} \Theta (f +h _k)(h _1, \ldots , h _{k-1})(h)
-
D^{k-1} \Theta (f )(h _1, \ldots , h _{k-1})(h )
\\
=
\Theta (f +h _k) (h h _1 \ldots  h _{k-1})-\Theta (f ) (h h _1 \ldots h _{k-1}).
\end{multline*}
From (\ref{representacaoemseriesdepotencias}) follows that
\begin{multline*}
D^{k-1} \Theta (f +h _k)(h _1, \ldots , h _{k-1})(h)
-
D^{k-1} \Theta (f )(h _1, \ldots , h _{k-1})(h )
\\
=
\sum \limits_{n=1} \limits^{\infty}
\frac{1}{n!}\Theta (f)(h h _1 \ldots h _{k-1} [h _k]^n).
\end{multline*}
Clearly, the above equation shows that
\begin{multline*}
D^{k-1} \Theta (f +h _k)(h _1, \ldots , h _{k-1})
-
D^{k-1} \Theta (f )(h _1, \ldots , h _{k-1})
\\
= 
\Theta (f )(\cdot)h _1 \ldots h _{k-1}h _k)
+
\mathcal{O}_k(h _k)(h _1 \ldots h _{k-1}),
\end{multline*}
where $\mathcal{O}_k(h _k)$ is the element of $L^{k-1}$ given by
$$
\mathcal{O}_k(h _k)(h _1 \ldots h _{k-1})
=
\sum \limits_{n=2} \limits^{\infty}
\frac{1}{n!}\Theta (f)((\cdot) h _1 \ldots h _{k-1} [h _k]^n).
$$
These upper bounds together with the inequality~(\ref{dominando})
enable us to conclude that the map 
$( h _1 \ldots h _k )\mapsto  \Theta (f )((\cdot) h _1 \ldots h _k )$ 
is an element of $L^k$.
The inequality~(\ref{dominando}) and the definition of $ \mathcal{O}_k(h _k)$ 
give us the upper bound 
\[
\| \mathcal{O}_k(h _k)(h _1 \ldots h _{k-1}) \| _{V} 
\leq 
\sum \limits_{n=2}\limits^{\infty} 
\frac{1}{n!}
\| \mathscr L _{f}\| _{V}  \| h _1 \| \ldots \| h _{k-1} \|(\| h _k\|)^n
\]
and consequently 
$(1/\| h _k \|)\| \mathcal{O}_k (h _k)\| _{V} \to 0$, 
when $\| h _k \|  \to 0$. 
Therefore, 
$D^k \Theta (f)(h _1, \ldots h _k)=\Theta (f )((\cdot)h _1 \ldots h _k)$
and the claim is proved.

By using the Claim 1 and the above estimates for the
remaining the analyticity of the mapping 
$\mathcal{K}\ni f\mapsto \mathscr{L}_f\in V$
follows.

\noindent {\bf Analitycity of } $\Theta^{*}.$ 
Let $f,g$ and $h$ be potentials in $\mathcal{K}$ 
and $\phi^{*}\in  \mathcal{K}^{*}$.
From the expansion \eqref{representacaoemseriesdepotencias} 
for $\Theta(f+h)$ we get 
\begin{eqnarray}\label{analit. adjunto 1}
\Theta^{*}(f+h)(\phi^{*})g
=
\phi^{*}(\Theta(f+h)(g))
&=&
\phi^{*}\left(\sum_{n=0}^{\infty}\dfrac{1}{n!}\Theta(f)(g [h]^n)(\cdot)\right)
\nonumber
\\
&=&
\sum_{n=0}^{\infty}\dfrac{1}{n!}\phi^{*}(\Theta(f)(g [h]^n)(\cdot))
\nonumber
\\
&=&
\sum_{n=0}^{\infty}\dfrac{1}{n!}(\Theta^{*}(f)\phi^{*})(g [h]^n)(\cdot).
\end{eqnarray}

\noindent {\bf Claim 2.} Consider  the derivative map 
$D\Theta^{*}:\mathcal{K}\to L(\mathcal{K},L(\mathcal{K}^*,\mathcal{K}^*)),$ 
then for any $f\in \mathcal{K}$ and $h\in \mathcal{K}$ we have that  
$D\Theta^{*}(f)(h):  \mathcal{K}^{*}\to  \mathcal{K}^{*}$ 
is given by
$
(D\Theta^{*}(f)(h)(\phi^{*}))g=(\Theta^{*}(f)(\phi^{*}))(gh).
$
Indeed, consider 
$\mathcal{O}:\mathcal{K}\to L(\mathcal{K}^*,\mathcal{K}^*) $ 
defined by
$
\mathcal{O}(h)(\phi^{*})
= 
\sum_{n=2}^{\infty}(1/n!)\Theta^{*}(\phi^{*})((\cdot)[h]^n)
$
then  we have
\begin{multline}
\label{analit. adjunto 2}
\|\mathcal{O}(h)\|_{L(\mathcal{K}^*,\mathcal{K}^*)}
=
\sup_{\|\phi^{*}\|_{*}=1}
\|\mathcal{O}(h)(\phi^{*})\|_{*}
=
\sup_{\|\phi^{*}\|_{*}=1}
\left\|
\sum_{n=2}^{\infty}\frac{1}{n!}
\Theta^{*}(f)(\phi^{*})((\cdot)[h]^n)
\right\|_{*}
\\
\leq
\sup_{\|\phi^{*}\|_{*}=1}
\sum_{n=2}^{\infty}\frac{1}{n!}
\underbrace{
	\left\|
		\Theta^{*}(f)(\phi^{*})((\cdot)[h]^n)
	\right\|_{*}
}_{I}.
\end{multline}

Next step is to upper bound the quantity $I$.
\begin{eqnarray}\label{analit. adjunto 3}
I
=
\|\;\Theta^{*}(f)(\phi^{*})((\cdot)[h]^n)\;\|_{*}
&=&
\sup_{\|g\|=1}
\|\Theta^{*}(f)(\phi^{*})((g)[h]^n)\|
\nonumber
=
\sup_{\|g\|=1}
\|\phi^{*}(\Theta(f)((g)[h]^n))\|
\nonumber
\\[0.3cm]
&\leq&
\|\phi^{*}\|_{*}
\sup_{\|g\|=1}
\|\Theta(f)((g)[h]^n)\|
\nonumber
\leq
\|\phi^{*}\|_{*}\|
\Theta(f)
\|_V\sup_{\|g\|=1}\| g[h]^n \|
\nonumber
\\[0.3cm]
&\leq&
\|\phi^{*}\|_{*}
\|\Theta(f)\|_{V}
\sup_{\|g\|=1}
\|g\|
\|[h]^n\|
\nonumber
\leq
\textnormal{const.} \|h\|^{n}.
\end{eqnarray}
By replacing (\ref{analit. adjunto 3}) in (\ref{analit. adjunto 2}) 
we get that 
\[
\|\mathcal{O}(h)\|_{L(\mathcal{K}^*,\mathcal{K}^*)}
\leq 
\sum_{n=2}^{\infty}
\frac{1}{n!}
\textnormal{Const.} 
\|h\|
\]
therefore  
$ 
\|\mathcal{O}(h)\|_{L(\mathcal{K}^*,\mathcal{K}^*)}/ \|h\|\to 0
$
when $\|h\|\to 0.$

It is possible to show that the higher orders 
derivatives  
$
D^{k}\Theta^{*}:
\mathcal{K}\to L(\mathcal{K}^k, L(\mathcal{K}^*,\mathcal{K}^*))
$ 
for $k\geq 2$, are given by the following expression
\begin{align*}
D^{k}\Theta^{*}(f)(h_1,\ldots, h_k)\phi^{*}(g)
= 
\Theta^{*}(f)\phi^{*}(h_1\cdots h_k g)
=
\phi^{*}(\mathscr{L}_f(h_1\cdots h_kg)).
\end{align*}
The proof is similar to the previous one and so 
it will be omitted.

\section{The Ruelle Theorem on the Walters Space}

To simplify the notation for any $f\in C(\Omega)$ 
and $x,y\in\Omega$, we write 
$$
S_nf(x) \equiv f(x)+f(\sigma(x))+\ldots+f(\sigma^{n-1}(x))
\quad
\text{and}
\quad
d_n(x,y) \equiv \max_{0\leq k<n} d_{\Omega}(\sigma^k x,\sigma^k y).
$$

\begin{defnc}\label{def-Walters-class}
We say that a continuous function $f:\Omega \to \mathbb{R}$ 
is in the Walters class if given $\epsilon>0$ there exists $\eta>0$ such that 
\begin{equation}\label{W.condition}
\forall n\geq 1,\forall x,y\in \Omega,~~ 
d_n(x,y)
\leq 
\eta \Longrightarrow |S_nf(x)-S_nf(y)|
\leq 
\epsilon
\end{equation}
The space of all continuous function satisfying the 
above condition is denoted by $W(\Omega)$.
\end{defnc}
If a continuous function $f:\Omega\to\mathbb{R}$ 
satisfies the condition \eqref{W.condition}, we say 
that $f$ satisfies the \emph{ Walters condition}.

\begin{defnc}\label{def-Walters-class}
Consider a continuous function $f:\Omega \to \mathbb{R}$ and define $C_f(x,y)$ by 
\begin{equation}\label{Weak W.condition}
C_f(x,y)=\sup_{n\geq 1}\sup_{\textnormal{\ba} \in M^n}S_nf(\textnormal{\ba} x)-S_nf(\textnormal{\ba} y).
\end{equation}
We say that $f$ satisfies the {\bf weak Walters condition} if $C_f(x,y)\to 0$ when $d(x,y)\to 0$.
\end{defnc}

\noindent \textbf{Example.} Consider the metric space
 $(M,d)$ were $M=[0,1]$ and $d=|\cdot|$. 
 Now let $f$ be the potential defined on  
$\Omega=M^{\mathbb{N}}$ by $f(x)=x_1$, i.e., 
$f$ depends only on the first coordinate. 
We claim that $f$ satisfies the weak Walters condition but 
not the (strong) Walters condition.
In fact, we have that $S_nf(\ba x)=S_nf(\ba y)=\sum_{i=1}^na_i$ 
 for any $\ba=(a_1,\ldots, a_n)\in M^n$.
Therefore $f$ clearly satisfies the weak Walters condition. 
Now we show that $f$ does not satisfy that Walters condition.
Indeed, consider $x=(0,0\ldots)$ the null vector
and $y=(\eta,\eta, \ldots)$ for a small $\eta$.
Notice that $d_{\Omega}(x,y)=d_n(x,y)=\eta.$
On the other hand, $S_nf(x)=0$ and $S_nf(y)=n\cdot \eta,$ then
$$
|S_nf(x)-S_nf(y)| = n\cdot \eta.
$$
From this is clear that $f$ does not satisfy the Walters condition.

On the other hand the opposite implication is always truth:
\begin{prop}\label{prop-strong-implies-weak-WC}
Let be $f\in C(\Omega)$ satisfying the Walters condition
 then $f$ satisfies the weak Walters condition.
\end{prop}

\begin{proof}
 Let $f\in W(\Omega),$ 
then by definition  given $\epsilon>0$ 
arbitrarily there exists $\eta>0$ such that
  $\forall n\geq 1, \forall z,w\in \Omega$
   with $d_n(z,w)\leq \eta$  we have  
    $|S_nf(z)-S_nf(w)|\leq \epsilon.$ 
    Note that $d_n(\ba x,\ba y)\leq d(x,y)$ 
for any $\ba\in M^n$. 
Therefore  
$
d(x,y)
\leq 
\eta
\Rightarrow 
d_n(\ba x,\ba y)
\leq 
\eta 
\Rightarrow 
|S_nf(\ba x)-S_nf(\ba y)|
\leq 
\epsilon,~~\forall \ba \in M^n,~~\forall n\geq 1.
$ 
By taking the supremum over all $\ba \in M^n$ and  $n\geq 1$  
the result follows.
\end{proof}

The space $W(\Omega)$ is clearly a linear space. Let $S$ 
denote the expansivity constant of the mapping $\sigma.$  
In \cite{Bousch} it was shown that for $s\in (0,S)$ 
the following  expression  
$$
\|f\|_{W_s}
=
2\,\|f\|_{0}
+\sup_{n\geq 1}\ 
\max_{d_n(x,y)\leq s}|\,S_nf(x)-S_nf(y)\,|
$$ 
defines a family of equivalent norms and 
$(W(\Omega), \|\cdot \|_{W_s})$ is a Banach Space.
 Since  the family of norms $(\|\cdot \|_{W_s})_{0<s<S}$ 
provides the same topology, there is no lost of generality 
in taking a particular value $s\in (0,S)$ and 
develop the theory with the norm 
$\|\cdot\|_W\equiv \|\cdot \|_{W_s}$.

En route to the proof of this work's main
theorem, we need an extra structure of this space which 
is the structure of the Banach algebra. 
This is the content of the next lemma.

\begin{lema} \label{Banach. Alg.}
The space $W(\Omega)$ with the norm $\|\cdot\|_W$ 
is a Banach algebra over $\mathbb{R}$. i.e., 
$W(\Omega)$ is a real Banach space and for all 
$f, g \in W(\Omega)$ we have
$\|fg\|_W\leq \|f\|_W\|g\|_W$.
\end{lema}

\begin{proof}
Let $f,g\in W(\Omega)$ and define $I\equiv I(f,g)$ by 
\begin{eqnarray*}
I&\equiv&\sup_{n\geq 1}\max_{d_n(x,y)\leq s}|S_n(fg)(x)-S_n(fg)(y)|
\\
&
=
&
\sup_{n\geq 1}\max_{d_n(x,y)\leq s}
\left|
\sum_{j=1}^{n-1}(fg)\circ \sigma^j(x)
-\sum_{j=1}^{n-1}(fg)\circ \sigma^j(y)
\right|
\\[0.2cm]
&
=
&
\sup_{n\geq 1}\max_{d_n(x,y)\leq s}
\left|
\sum_{j=1}^{n-1}(fg)\circ \sigma^j(x) 
-\sum_{j=1}^{n-1}f\circ\sigma^j(x)g\circ\sigma^j(y)
\right.
\\
&&
\left.
\hspace*{7.5cm}
+\sum_{j=1}^{n-1}f\circ\sigma^j(x)g\circ\sigma^j(y)
-\sum_{j=1}^{n-1}(fg)\circ \sigma^j(y)
\right|.
\end{eqnarray*}
By applying the Triangular Inequality we get that
\begin{multline*}
I
\leq
\sup_{n\geq 1}\max_{d_n(x,y)\leq s}
\left|
\sum_{j=1}^{n-1}f\circ \sigma^j(x)
\Big[ g\circ\sigma^j(x)-g\circ \sigma^j(y)\Big]
\right|
\\
+
\sup_{n\geq 1}\max_{d_n(x,y)\leq s}
\left|
	\sum_{j=1}^{n-1}g\circ \sigma^j(y)
		\Big[ f\circ\sigma^j(x)-f\circ\sigma^j(y) \Big]
\right|
\\[0.3cm]
\leq
\|f\|_{0}
\cdot 
\sup_{n\geq 1}
\max_{d_n(x,y)\leq s}
|S_ng(x)-S_ng(y)|+\|g\|_{0}
\cdot 
\sup_{n\geq 1}
\max_{d_n(x,y)\leq s}
|S_{n}f(x)-S_nf(y)|.
\end{multline*}
The last upper bound readily implies that
\[
\|fg\|_{W}
\leq 
2\|fg\|_{0}
+
\|f\|_{0}\,
\sup_{n\geq 1}
\max_{d_n(x,y)\leq s}
|S_ng(x)-S_ng(y)|
+
\|g\|_{0}\,
\sup_{n\geq 1}
\max_{d_n(x,y)\leq s}
|S_f(x)-S_nf(y)|. 
\]
On the  other hand, 
it follows from the definition of $\|\cdot\|_W$ that
\begin{multline*}
\|f\|_{W}
\cdot 
\|g\|_{W}
=
4 \|f\|_{0} \|g\|_{0}+2\|f\|_{0}
\cdot 
\sup_{n\geq 1}
\max_{d_n(x,y)\leq s}
|S_ng(x)-S_ng(y)|
\\
+2\|g\|_{0}
\cdot 
\sup_{n\geq 1}
\max_{d_n(x,y)\leq s}
|S_f(x)-S_nf(y)|
\\
+
\sup_{n\geq 1}
\max_{d_n(x,y)\leq s}
|S_ng(x)-S_ng(y)|
\cdot  
\sup_{n\geq 1}
\max_{d_n(x,y)\leq s}
|S_{n}f(x)-S_nf(y)|.
\end{multline*}
This identity and the previous estimates 
readily implies that 
$\|fg\|_{W}\leq \|f\|_{W}\cdot \|g\|_{W}$.
\end{proof}

\begin{prop}		
If $f\in W(\Omega)$, then $\mathscr{L}_f\big(W(\Omega)\big) \subset W(\Omega)$.
\end{prop}

\begin{proof}
We claim that for any fixed $a\in M$, if $f\in W(\Omega)$ then the function 
$x\mapsto f(ax)$ also belongs to $W(\Omega)$.
In fact, given $\epsilon>0$ we choose $\eta>0$ such that
the Walters condition \eqref{W.condition} is satisfied for $f$. 
 Note that $d_n(x,y)\leq \delta \Rightarrow d_n(ax,ay)\leq d_n(x,y)\leq \eta$.
From the definition we have that $d_n(x,y)\leq \eta$ implies 
 $|S_nf(ax)-S_nf(ay)|
 \leq  \epsilon$ 
 for all $n>0$ and therefore the claim is proved.

The next step is to prove that   
the function $r:\Omega\to\mathbb{R}$ given by $r(x)= \int _M h(ax) d\mu(a)$
belongs to $W(\Omega)$ whenever $h\in W(\Omega)$.
From the previous claim follows that
$x\mapsto h(ax)$ is in $W(\Omega)$ since $h\in W(\Omega).$
Given $\epsilon>0$ we now choose $\eta>0$ such that
the condition \eqref{W.condition} is satisfied for $x\mapsto h(ax)$.
Since $d_n(x,y)\leq \eta\Rightarrow d_n(ax,ay)\leq \eta$, 
the Walters condition for $r$ follows from the inequality
\begin{align*}
|S_nr(x)-S_nr(y)|
&\leq
\int_{M} 
\left|
S_nh(ax)-S_nh(ay)
\right|
\, d\mu(a)
\leq
\epsilon \cdot \mu(M).
\end{align*}

By hypothesis the potential $f\in W(\Omega)$, 
since the Walters space is a Banach algebra we have that 
$\exp(f)\in W(\Omega)$. For the same reason, for any $\varphi\in W(\Omega)$ 
we have $\varphi\cdot \exp(f) \in W(\Omega)$.
As argued above, for any $a\in M$,
the mapping $x\mapsto \varphi(ax)\cdot \exp(f(ax))$
belongs to $W(\Omega)$. Using the result proved 
above for the function $r$, with 
$h(x) =\varphi(x) \exp(f(x))$, it follows that 
the mapping 
\[
x
\mapsto 
\int_{M}\varphi(ax) \exp(f(ax))\, d\mu(a)
\equiv 
\mathscr{L}_{f}(\varphi)(x)
\]
is in the Walters space for any $\varphi \in W(\Omega)$ 
which finishes the proof.
\end{proof}

\section{The Ruelle Theorem On Walters space}

The proof of this version of the Ruelle Theorem is 
inspired in the original proof presented in Walters 
\cite{PW}.

Let be $f$ potential in $C(\Omega)$.  Consider the function $C_f$ given by 
\[
C_{f}(x,y)
=
\sup_{n\geq 1}\,
\sup_{\textbf{a}\in M^n}\,
\sum_{i=0}^{n-1}
(f(\sigma^i(\textbf{a}x))-f(\sigma^i(\textbf{a}y)).
\]
We say that  $f$ satisfies the weak Walters condition 
if $C_{f}(x,y)\to 0$, when $d_{\Omega}(x,y)\to 0$. 

To prove the Ruelle Theorem we concentrate attention 
on a certain subclass of $C(\Omega)$ which is given by 
\[
G_0(\Omega)
=
\left\{
g\in C(\Omega): 
g>0~~ \textnormal{and}~~
\int_{M}g(ax)\, d\mu(a)
=
1
~\forall x\in 
\Omega
\right\}.
\]
If $f:\Omega\to\mathbb{R}$ is a potential given by $f=\log g$, 
where $g\in G_0(\Omega)$, then 
the weak Walters condition for $f$ can be rephrased in 
terms of $g$ by saying that  
\[
D_g(x,y)
=
\sup_{n\geq 1}
\sup_{\ba \in M^n}
\prod_{i=0}^{n-1}
\dfrac{g(\sigma^i \ba x)}{g(\sigma^i \ba y)}
\]
exists, is bounded for a constant 
$D_g$ and $D_g(x,y)\to 1$ when  $d_{\Omega}(x,y)\to 0$. 
Equivalently:
\[
D^{\star}_g(x,y)
=
\sup_{n\geq 1}
\sup_{\ba \in M^n}
\left|
	\prod_{i=0}^{n-1}
	\dfrac{g(\sigma^i \ba x)}
			{g(\sigma^i \ba y)}
	-1
\right|
\leq 
D_g-1
\]
for all  
$x,y$ with $d_{\Omega}(x,y)<\epsilon_0$ 
and  
$D^{\star}_g(x,y)\to 0$ when $d_{\Omega}(x,y)\to 0$.

\begin{teo}\label{pre-walters}
Let $g\in G_0(\Omega)$ be a function such that $\log g$ 
satisfies the weak Walters condition. 
Then there is a  probability measure 
$\nu:\mathscr{B}(\Omega)\to [0,1]$ such that 
\[
\mathscr{L}^n_{\log g}
\varphi\stackrel{\|\cdot\|_{0}}{\longrightarrow}
\nu(\varphi)
\]
for all $\varphi\in C(\Omega)$. 
Moreover $\nu$ 
is the unique probability measure satisfying 
$\mathscr{L}^{*}_{\log g}\nu=\nu$.
\end{teo}

\begin{proof}
The proof is based on a simple 
modification of the arguments given in 
\cite{PW78} and it is presented here 
 for the reader's convenience.

Let us introduce the temporary notation $\mathscr{L}$
for $\mathscr{L}_{\log g}$.
We begin by proving that   
$\{\mathscr{L}^n\varphi,~n\geq 0\}$  is equicontinuous
family for any fixed $\varphi$ satisfying the weak Walters condition. Indeed, 
from the definition of the Ruelle operator we have 
\begin{multline*}
|\mathscr{L}^n\varphi(x)-\mathscr{L}^n\varphi(y)|
\leq
\left|
	\int_{M^n}
	\big[
	\exp(S_n\log g(\ba x) )\varphi(\ba x)
	-
	\exp(S_n\log g(\ba y) )\varphi(\ba y)
	\big]
	\prod_{i=1}^{n}d\mu(a_{i})
\right|
\\ 
\leq
\left|
	\int_{M^n}
	\prod_{i=0}^{n-1}
	g(\sigma^{i} (\ba x))
	[\varphi(\ba x)-\varphi(\ba y)]
	\prod_{i=0}^{n}d\mu(a_{i})
\right|
\\
+
\left|
	\int_{M^n}
	\varphi(\ba y)
	\left[
	  \prod_{i=0}^{n-1}g(\sigma^{i} (\ba x))
	  -
	  \prod_{i=0}^{n-1}g(\sigma^{i} (\ba y))
	 \right]
	 \prod_{i=0}^{n}d\mu(a_{i})
\right|
\end{multline*}
The two terms in the rhs above can be bounded by 
\begin{multline*}
\sup_{\ba \in M^n}\{|\varphi(\ba x)-\varphi(\ba y)| \}
\left|
   \int_{M^n}\
   \prod_{i=0}^{n-1}
   g(\sigma^{i} (\ba x))
   \prod_{i=1}^{n} d\mu(a_i)
\right|
\\
+
\|\varphi\|_0 
\int_{M^n}\
\prod_{i=0}^{n-1}
g(\sigma^{i} (\ba y))
\left|
   \dfrac{\prod_{i=0}^{n-1} g(\sigma^{i} (\ba x))}
   {\prod_{i=0}^{n-1}g(\sigma^{i} (\ba y))}
   -1
\right|
\ \ 
\prod_{i=1}^{n} d\mu(a_i).
\end{multline*}
Since $g\in G_{0}(\Omega)$ it follows from 
the Fubini Theorem that the iterated integral 
on the first term above is equal to one. 
The second term can be bound similarly by 
using the definition of $D^{\star}_g(x,y)$,
which give us the following inequality 
\begin{align*}
|\mathscr{L}^n\varphi(x)-\mathscr{L}^n\varphi(y)|
\leq 
\sup_{\ba \in M^n}\{|\varphi(\ba x)-\varphi(\ba y)|\}
+
\|\varphi\|_0 \cdot D^{\star}_g(x,y).
\end{align*}
Since $\varphi$ is a continuous function and 
$\log g$ satisfies the weak Walters condition  
the above inequality ensures that the family 
$\{\mathscr{L}^{n}\varphi,~n\geq 0\}$ 
is equicontinuous. 

Recalling that $g\in G_{0}(\Omega)$ we get 
from the definition of the Ruelle operator
for all $n\in\mathbb{N}$ that $\|\mathscr{L}^n\varphi\|_0\leq \|\varphi\|_0~$ 
for any $\varphi\in C(\Omega)$. This inequality 
implies that the closure in the uniform topology of 
$\{\mathscr{L}^{n}\varphi,~n\geq 0\}$ 
is uniformly bounded in  $C(\Omega)$. 
Therefore we can apply the Arzela-Ascoli Theorem 
for the family $\{\mathscr{L}^{n}\varphi,~n\geq 0\}$
to guarantee that there exist a subsequence $(n_i)\subset \mathbb{N}$ 
and function $\varphi^{\star}\in C(\Omega)$ 
so that  
$
\mathscr{L}^{n_i}\varphi \longrightarrow \varphi^{\star}
$
uniformly.

Let us proceed by showing that $\varphi^{\star}$ 
is a constant function.
Notice that the identity
$\mathscr{L}(1)=1$ implies the following inequalities  
$\min (\varphi)
\leq 
\min (\mathscr{L}(\varphi))
\leq 
\cdots 
\leq 
\min(\varphi^{\star}).
$
\noindent {\bf Claim 1.} 
For any $k\in\mathbb{N}$ we have
$
\min (\mathscr{L}^k \varphi^{\star})
=
\min (\varphi^{\star})
$. 
Indeed, we have that  
$
\min (\mathscr{L}^k \varphi^{\star})
=
\min(\mathscr{L}^k(\lim \mathscr{L}^{n_i}\varphi))
=
\min(\lim \mathscr{L}^{n_i+k}\varphi)
=
\lim (\min(\mathscr{L}^{n_i+k}\varphi))
=
\min(\varphi^{\star}),
$
where the last equality  it follows from the 
monotonicity of the sequence 
$\min \mathscr{L}^k \varphi$ and 
$\min\mathscr{L}^{n_i}\varphi \longrightarrow \min \varphi^{\star}$.
Given $\epsilon>0$ choose  $x\in \Omega$ and $N\in\mathbb{N}$ 
such that 
$
\min (\mathscr{L}^N\varphi^{\star})
=
\mathscr{L}^N\varphi^{\star}(x)
$ 
and 
$\{\ba x,~\ba \in M^N\}$ is $\epsilon$-dense in $\Omega$. 
\noindent {\bf Claim 2.}  
For all $y\in \sigma^{-N}x$ we have
$\varphi^{\star}(y)=\min (\varphi^{\star})$.
From the Claim $1$ and the choice of $x$, we have 
$
\mathscr{L}^N \varphi^{\star}(x)
=
\min (\varphi^{\star}).
$
Let $z\in \Omega$ be such that 
$\varphi^{\star}(z)=\min(\varphi^{\star} )$,
then 
\[
\int_{M^N}
g(\ba x)g(\sigma\ba x)\cdots g(\sigma^{N-1}\ba x)
\varphi^{\star}(\ba x)
\, \prod_{i=1}^{N} d\mu(a_i)
=
\mathscr{L}^N\varphi^{\star}(x)
=
\min(\varphi^{\star} )
=
\varphi^{\star}(z).
\]
By using the identity $1=\mathscr{L}^N_{\log g}1(x)$,
it follows from the above equation that 
\begin{align*}
0
=
\int_{M^N}
g(\ba x)g(\sigma\ba x)
\cdots 
g(\sigma^{N-1}\ba x)
\Big[ \varphi^{\star}(\ba x)-\varphi^{\star}(z) \Big]
\, \prod_{i=1}^{N} d\mu(a_i)
\end{align*}
By using the continuity of $\varphi^{\star}$ 
and the assumption $\mathrm{supp}(\mu)=M$ it is easy 
to see that $\varphi^{\star}(\ba x)=\varphi^{\star}(z)$
for any $\ba\in M^N$. 
Since  $\varphi^{\star}$ is 
continuous and constant over
$\{\ba x,~\ba \in M^N\}$, 
which is $\epsilon$-dense in $\Omega$,
follows that $\varphi^{\star}$ is a constant function.

We now shown the existence and uniqueness of the fixed point 
for $\mathscr{L}^{*}\equiv \mathscr{L}^{*}_{\log g}$. 
Define the linear functional $F: C(\Omega)\to\mathbb{R}$ 
by  $F(\varphi)= \varphi^{\star}$.  
The functional $F$ sends the cone of positive continuous functions 
to itself and satisfies $F(1)=1$. Then it follows from the 
Riesz-Markov Theorem that there exists an unique 
Borel probability measure $\nu\in \mathcal{M}(\Omega)$
that represents $F$.  
It is a simple matter to show that $\mathscr{L}^{*}\nu=\nu$. 
For the uniqueness suppose that there exists another 
probability measure $\gamma\in \mathcal{M}(\Omega)$ 
such that $\mathscr{L}^*\gamma=\gamma$. 
Of course, $(\mathscr{L}^*)^n\gamma=\gamma$ for every $n\in\mathbb{N}$
and  
\[
\int_{\Omega} 
\varphi\, 
d\gamma
=
\int_{\Omega} 
\varphi
\, d\big[ (\mathscr{L}^*)^n \gamma \big] 
=
\int_{\Omega} 
\mathscr{L}^n\varphi
\, d\gamma 
=
\int_{\Omega} 
\lim_{n\to\infty}\mathscr{L}^n\varphi
\, d\gamma
=
\int_{\Omega}
\varphi^{\star}
\, d\gamma
=
\varphi^{\star}
=
\int_{\Omega}
\varphi
\, d\nu.
\]
Since $\varphi\in C(\Omega)$ is arbitrary it follows that 
$\gamma=\nu$.
\end{proof}

\begin{lema}\label{cota-inf-sn-a}
Let $f$ satisfying the weak Walters condition, then  
$\forall \epsilon>0$, there exists $N>0$ and 
$a\in \mathbb{R}$ such that 
$\forall x, y\in \Omega,~~\exists w\in  
\sigma^{-N}x\cap  B_{d_{\Omega}}(y,\epsilon)$ 
with $S_Nf(w)\geq a$.
\end{lema}

\begin{proof}
See \cite{PW78} page 126.
\end{proof}

\begin{lema}\label{lema-auto-medida-L-star}
Let $f \in  C(\Omega)$ be a potential.
Then there exists a real number $\lambda>0$ 
and a Borel probability measure $\nu\in \mathcal{M}(\Omega)$  
such that $\mathscr{L}_{f}^{*}\nu=\lambda\nu$.
\end{lema}

\begin{proof}
The  mapping 
$
\gamma 
\mapsto 
\mathscr{L}_{f}^{*}\gamma/(\mathscr{L}_{f}^{*}\gamma)(1)
$ 
define a continuous function from $\mathcal{M}(\Omega)$ 
to itself. The Schauder-Tychonoff fixed point theorem 
ensures the existence of a fixed point $\nu$ for this mapping.  
By taking $\lambda=(\mathscr{L}_{f}^{*}\nu)(1)$ the 
theorem follows.
\end{proof}

We are now able to prove the main
theorem of this section which is the 
Ruelle Theorem for Walters potentials
defined over an infinite cartesian 
product of general metric compact spaces.

\begin{teo}\label{Ruelle-Walters}
	Let $f$ be a potential satisfying the weak Walters condition  and consider the Ruelle operator 
	$\mathscr{L}_{f}:C(\Omega)\to C(\Omega)$ 
	associated to $f$.  
	Then there are a real number
	$
	\lambda_{f} >0,
	$ 
	a strictly positive continuous function $h_{f}$
	and a {\bf unique} Borel probability measure $\nu_{f}$ 
	such that 
	\begin{itemize}
		\item[i)] 
		$
		\mathscr{L}_{f} h_f
		=
		\lambda_{f} h_{f}~$, 
		$\mathscr{L}_{f} ^{*}\nu_{f} 
		=
		\lambda_{f} \nu_{f} 
		$.
		
		\item[ii)] 
		For any $\varphi\in C(\Omega)$ we have 
		\[
		\left\| 
		\lambda_{f} ^{-n}
		\mathscr{L}^n_{f}\varphi-h_{f} 
		\int_{\Omega} \varphi\, d\nu_{f} 
		\right\|_0
		\longrightarrow
		0,
		\qquad
		\text{when}\ n\to\infty.
		\]
		
	\end{itemize}
\end{teo}

\begin{proof}
The proof will be divided into three steps.
\\
\noindent {\bf Claim 1.} Let $\nu$ and $\lambda$ as given
by the Lemma \ref{lema-auto-medida-L-star}. Then 
for any $f$ satisfying the weak Walters condition and $\epsilon_0>0$ the set  
$$
\Lambda
=
\{ 
\varphi\in C(\Omega) : 
\varphi\geq 0,\ 
\nu(\varphi)=1
~~\textnormal{and}~~
\varphi(x)
\leq 
\exp{(C_{f}(x,y))}\varphi(y)
~~\textnormal{if}~
d_{\Omega}(x,y)
<\epsilon_0
\}
$$
is convex, closed, bounded and uniformly equicontinuous. 

Let us first prove that $\Lambda$ is not empty. 
Indeed, for any $x,y\in\Omega$ we have
\begin{multline*}
\mathscr{L}_{f}1(x)
=
\int_{M}e^{f(ax)}\, d\mu(a)
=
\int_{M} e^{f(ay)}e^{f(ax)-f(ay)}\, d\mu(a)
\\
\leq 
\exp{ \left( \sup_{a\in M}  f(ax)-f(ay) \right) }
\int_{M} e^{f(ay)}\, d\mu(a)
\leq 
\exp{(C_{f}(x,y))} \mathscr{L}_{f}1(y).
\end{multline*}
The set $\Lambda$ is clearly closed and convex.  
Now we shall prove that $\Lambda$ is a bounded set. 
Let $x,y\in \Omega$. 
By the Lemma \ref{cota-inf-sn-a} 
given $\epsilon>0$ and $a\in \mathbb{R}$
there is $N\in\mathbb{N}$ and $y_0=\ba_0 x$, 
where $\ba_{0} = a_1\ldots a_{N}$, 
such that $d_{\Omega}(y_0,y)<\epsilon$ and  $S_Nf(y_0)\geq a$. 
Given $\delta_1>0$ it follows from the continuity of $\varphi$ that
we can choose $\delta>0$ such that for any point $\ba$ in the closed 
ball $B[\ba_{0}, \delta]\subset M^N$ we have $S_Nf(\ba x)\geq a-\delta_1$.  In 
particular, we can choose $\delta$  such that $B[y_0,\delta]\subset B[y,\epsilon]$.
Therefore it follows from the definition of the Ruelle operator and 
the choice of $\delta$ that 
\begin{align*}
\mathscr{L}^N\varphi(x)
&=
\int_{M^N}e^{S_Nf(\ba x)}
\varphi(\ba  x) 
\prod_{i=1}^{N} d\mu(a_i)
\\
&=
\int_{M^N\setminus B[\ba_{0}, \delta]}
e^{S_Nf(\ba x)}\varphi(\ba  x)
\prod_{i=1}^{N} d\mu(a_i)
+
\int_{ B[\ba_{0}, \delta]}
e^{S_Nf(\ba x)}\varphi(\ba  x)
\prod_{i=1}^{N} d\mu(a_i)
\\
&\geq
\int_{ B[\ba_{0}, \delta]}e^{S_Nf(\ba x)}
\varphi(\ba  x)
\prod_{i=1}^{N} d\mu(a_i)
\\[0.2cm]
&\geq
\mu\times\ldots\times\mu (B[\ba_{0}, \delta])
e^{a-\delta_1}\varphi(w_0),
\end{align*}
where $w_0$ minimizes the function 
$(a_1,\ldots, a_N) \mapsto \varphi(a_1\ldots a_N x)$ in $B[y,\delta]$.
Now observe that $w_0\in B[y,\epsilon]$, 
by using the compactness of $\Omega$ and 
the definition of $\Lambda$
we get for all $x,y\in \Omega$ the following inequality 
$
\varphi(y)
\leq 
Const\cdot  
e^{\delta_1-a}\mathscr{L}^N\varphi(x).
$
Recalling that $\mathscr{L}_{f}^{*}\nu=\lambda\nu$
and $\nu(\varphi)=1$, 
we obtain by integrating both sides of the previous 
inequality that
$
\varphi(y)
\leq 
Const\cdot  e^{C+\delta_1-a}
\nu(\mathscr{L}^N\varphi(x))
= 
Const\cdot 
e^{C+\delta_1-a}\lambda^N.
$ Hence $\Lambda$ is bounded.
The uniform equicontinuity of $\Lambda$ 
is proved as in \cite[p. 129]{PW78} mutatis mutandis.

\medskip

\noindent {\bf Claim 2.} 
The operator $\lambda^{-1}\mathscr{L}_{f}$ maps $\Lambda$ into $\Lambda$.

Let $\varphi\in \Lambda$ and $x,y \in \Omega$ 
with $d_{\Omega}(x,y)<\epsilon_0$. Then 
\begin{align*}
\dfrac{1}{\lambda} \mathscr{L}\varphi(x)
&=
\dfrac{1}{\lambda}
\int_{M}
e^{f(ax)}\varphi(ax)\,
d\mu(a)
\\[0.2cm]
&\leq 
\dfrac{1}{\lambda} 
\int_{M} 
e^{f(ay)}\varphi(ay)
\underbrace{(e^{f(ax)-f(ay)}e^{C_{f}(ax,ay)})}_{\leq  e^{C_{f}(x,y)}}\,
d\mu(a)
\\[0.2cm]
&\leq
\dfrac{1}{\lambda}
\int_{M} 
e^{f(ay)}\varphi(ay) e^{C_{f}(x,y)}\,
d\mu(a)
\\[0.2cm]
&=
\dfrac{1}{\lambda}\mathscr{L}_{f}\varphi(y) 
e^{C_{f}(x,y)}
\end{align*}
where the inequality $e^{f(ax)-f(ay)}e^{C_{f}(ax,ay)}\leq e^{C_{f}(x,y)}$ 
is justified by observing that $C_f(ax,ay)$  is equal to 
$$
\sup_{n\geq 1}
\sup_{\ba \in M^n}
\{
	(f(\ba ax)-f(\ba a y))
	+(f(\sigma \ba x)-f(\sigma \ba a y))
	+\cdots 
	+(f(\sigma^{n-1}\ba a x)-f(\sigma^{n-1}\ba a y))
\}.
$$
From this is clearly that $C_{f}(ax,ay)+(f(ax)-f(ay))\leq C_f(x,y).$
\medskip

\noindent {\bf Claim 3.}  
If $g=e^{f}h/(\lambda h\circ T)$, then $g\in G_0(\Omega)$  
and $D^{\star}_g(x,y)\to 0$ when $d_{\Omega}(x,y)\to 0$.
The proof is similar to the one given by \cite[p. 130]{PW78}.

\medskip

The Claims  1 and  2 allow us to use the Schauder-Tychonoff fixed point theorem 
to obtain a fixed point  $h\in \Lambda$ for the operator $\lambda^{-1}\mathscr{L}$.
This fixed point $h$ satisfies $\mathscr{L}h=\lambda h$, $\nu(h)=1$ and 
$h(x)\leq e^{C_{f}(x,y)}h(y)$ whenever $d_{\Omega}(x,y)<\epsilon_0$.
We shall show that  $h>0$. 
Suppose that there exist some $x\in\Omega$ such that $h(x)=0$. 
For all $n\in\mathbb{N}$ we have $\mathscr{L}^nh(x)=\lambda^nh(x)=0$, 
so  $h$ must be $0$ on the set $\{\sigma^{-n}x, n\in \mathbb{N}\}$, 
which is dense, using that $\mu$ has full support we have $h\equiv 0$ 
contradicting the fact that $\nu(h)=1$.
\medskip

By the Theorem [\ref{pre-walters}] we have that  
$
\mathscr{L}^n_{\log g} \varphi
\stackrel{|\cdot|_0}{\longrightarrow}
\mu(\varphi)
$ for all  $\varphi\in C^{0}(\Omega)$, 
where $\mu\in \mathcal{M}(\Omega)$ is the fixed point 
of $\mathscr{L}^*_{\log g}$ in $\mathcal{M}(\Omega)$.
On the other hand  
$$
\dfrac{1}{\lambda^n} \mathscr{L}^n_{f}\varphi(x)
= 
h(x)(\mathscr{L}^n_{\log g} (\varphi/h)(x))
$$
then it follows that  
$
1/\lambda^{n}  
\mathscr{L}^n_{f}\varphi 
\stackrel{|\cdot|_0}{\longrightarrow} 
h\cdot \mu(\varphi/h)
$. 
We shall show that 
$\mu(\varphi/h)=\nu(\varphi)$. 
Let be $m\in \mathcal{M}(\Omega)$  
defined by $m(\varphi)=\nu(h\varphi)$. 
Then
$$
m(\mathscr{L}_{\log g}\varphi)
=
\nu(h\cdot \mathscr{L}_{\log g}\varphi )
=
\dfrac{1}{\lambda}=\nu(\mathscr{L}_{f}(\varphi\cdot h)
=
m(\varphi).
$$
\end{proof}

\section{Spectral Gap and Analyticity of the Pressure}

We mean by ``presence of the spectral gap" in the Ruelle
operator the existence of a single isolated eigenvalue of maximum 
modulus. 
The presence of the spectral gap in the Ruelle operator is the key property 
to prove analyticity of the pressure and 
also implies the exponential decay of correlations with 
respect to the Gibbs measures. These are classical results 
and very well known for H\"older potentials when the state space $M$ 
is finite. We shall here analyze the generalizations
of this results in the sense of the state space $M$ 
and the regularity of the potential.

The main difference between the Ruelle Theorem operator
for $f\in W(\Omega)$ and 
for $f\in C^{\gamma}(\Omega)$ is 
the fact that in the first case we do not have much 
information about the spectrum
of $\mathscr{L}_f$. 
In particular, it seems a hard problem to 
decide whether we have  
presence or absence of the spectral gap for
the Ruelle operator $\mathscr{L}_f$, which 
is crucial property to get deep understanding of the associated 
Gibbs measure.

\begin{teo}\label{Analiticadade Pressao}
Let $\mathcal{K}\subset C(\Omega)$ be an invariant subspace of $C(\Omega)$.  
A sufficient condition for the analyticity of the pressure is the analyticity of the map 
$\mathcal{K} \ni f \mapsto \mathscr{L}_{f} \in \mathcal{K}$
and the presence of the spectral gap in the 
spectrum of the Ruelle operator.
\end{teo}

In order to prove the Theorem \ref{Analiticadade Pressao} 
we shall use the following lemma, which seems to be  
a well known fact on the community. 
We decided to give a proof for this lemma 
either to keep the text as self contained as possible 
and because we were not able to find its reference.

In this paper we use the expression  
\emph{simple eigenvalue} to refer to an eigenvalue $\lambda$ 
of an operator $T:X\to X$ such that the image of the spectral projector 
$
\pi_L=\int_{\partial D}(\lambda I-L)^{-1}d\lambda
$
is an uni-dimensional subspace of $X$.

\begin{lema}\label{Lema tecnico}
Let $T:X\to X$  be  a bounded linear operator possessing an 
isolated simple eigenvalue $\lambda\in \mathbb{C}$.
Let $D$ be a closed disc centered at $\lambda$ such that 
$D \cap \mathrm{spec}(T)=\{\lambda\}$.   
Then there is a neighborhood $U$ of $T$ in $L(X,X)$ 
so that the mapping
$U\ni L\mapsto \lambda(L)\in \mathbb{C}$, 
where $\lambda(L)$ 
is the unique point in $\mathrm{spec}(L)\cap \mathrm{int}(D)$
is well defined 
and moreover this mapping is analytic.

\end{lema}

\begin{corollary}
For any fixed $0< \gamma \leq 1$ both mappings
$$
C^{\gamma}(\Omega)\ni f\mapsto h_f\in C^{\gamma}(\Omega)
\qquad \text{and} \qquad
C^{\gamma}(\Omega)\ni f\mapsto \nu_f\in (C^{\gamma}(\Omega))^{*}
$$
are analytic.
\end{corollary}

\begin{proof}
Choose $f$ arbitrarily in $C^{\gamma}(\Omega). $ 
Let  $U$ and $\tilde{U}$ be neighborhoods of 
$\mathscr{L}_{f}$ and  
$\mathscr{L}_f^{*}$, respectively 
such that for all $\mathscr{L}_{g}\in U$ and 
$\mathscr{L}^*_{g}\in \tilde{U}$ the images  of the spectral 
projectors  $\pi_{\mathscr{L}_g}$ and $\pi_{\mathscr{L}_g^{*}}$ 
are  the  one-dimensional spaces associated to the eigenvalue $\lambda_g$.
Consider a neighborhood $W\subset C^{\gamma}(\Omega)$ of $f$ 
so that $\mathscr{L}_g \in U$ 
and  $\mathscr{L}_g^{*} \in \tilde{U}$, respectively whenever $g\in W$. 
Therefore we have that  $\pi_{\mathscr{L}_g^{*}}\nu_f=c\cdot \nu_g$ 
then 
$
c
=
c\cdot \int_{\Omega} 1\, d\nu_g 
\equiv 
c\cdot \left<\nu_g,1\right>
=
\left<\pi_{\mathscr{L}_g^{*}}\nu_f,1\right>$ so 
$
\nu_g=\left<\pi_{\mathscr{L}_g^{*}}\nu_f,1\right>^{-1}\cdot \pi_{\mathscr{L}_g^{*}}\nu_f
$
Being the rhs of the last expression a composition of analytic functions 
follows that $g\mapsto \nu_g$ is analytic in a neighborhood of $f$. 
On the other hand,  
$\pi_{\mathscr{L}_g}h_f=C\cdot  h_g$ so 
by a suitable choice of the eigenfunction it follows 
that $C=C\cdot \left<\nu_g,h_g\right>=\left<\nu_g,\pi_{\mathscr{L}_g}h_f\right>$, 
therefore 
$
h_g=\left<\nu_g,\pi_{\mathscr{L}_g}h_f\right>^{-1}\cdot\pi_{\mathscr{L}_g}h_f
$
which is again a composition of analytic functions and so 
$g\mapsto h_g$ is analytic in a neighborhood of $f$. 
\end{proof}

Now we are ready to state and prove the main result of this section.
\begin{teo}\label{Teorema principal}
The function defined by  
$C^{\gamma}(\Omega) \ni f\rightarrow  P(f)\in \mathbb{R}$ is real analytic function.
\end{teo}

\begin{proof}
The proof is based on the analyticity of both functions
$C^{\gamma}(\Omega)\ni f \mapsto \mathscr{L}_{f}$
and  $U \ni T\mapsto \lambda(T)\in \mathbb{R}$.
The analyticity of the first map is the content of the Lemma (\ref{Lema principal}), 
whose detailed proof is given in \cite{ERR} Theorem 3.5.  
The analyticity of the second mapping it follows from 
the Theorem \ref{Ruelle-Perron}, 
which assures that $\mathscr{L}_{f}:C^{\gamma}(\Omega)\to C^{\gamma}(\Omega)$
has the spectral gap property and the Lemma \ref{Lema tecnico}.  
To finish the proof it is enough to observe that the mapping 
$C^{\gamma}(\Omega)\ni f\rightarrow  P(f)\in \mathbb{R}$ 
is given by $P(f)=\log\lambda_{f}$. Indeed, 
the rhs is the composition of the following 
analytic mappings: 
$\log$, $C^{\gamma}(\Omega)\ni f \mapsto \mathscr{L}_{f}$ 
and $U \ni T\mapsto \lambda(T)\in \mathbb{R}$.
\end{proof}

\noindent
{\bf Proof of the Lemma \ref{Lema tecnico}}.
The argument is based on the following two claims. 
\\ 
\noindent {\bf Claim 1.} 
There is a neighborhood $U$ of $T$ in $L(X,X)$ 
such that $\mathrm{spec}(L)\cap \partial D={\varnothing}$ 
for all $L\in U$.
This claim is proved by contradiction. Suppose that 
for each $n$ there exist 
$L_n \in B(T,1/n)\subset L(X,X)$
so that $\lambda_{L_n}\in \mathrm{spec}(L_n)
\cap \partial D$  and   
the operator $\lambda_{L_n}I-L_n$ 
is not invertible.
Since $\partial D$ is compact 
we can find a convergent subsequence 
$\{\lambda_{L_{n_{j}}}\}\subset \partial D$
so that $\lambda_{L_{n_{j}}}\to \lambda_L\in \partial D$. 
Since $L_n\to T$ we have that
$(\lambda_{L_{n_j}}I-L_{n_{j}})\to (\lambda_L I-T)$
in the strong topology.  
Since $(\lambda_L I-T)$ is invertible and 
the space of the invertible bounded linear 
operators is open we thus have a contradiction.
\\ 
\noindent {\bf Claim 2.} 
We can  shrink $U$  so that  for all $L\in U$
we have that $\mathrm{spec}(L)\cap \mathrm{int}(D)$ is  
a simple eigenvalue of $L$. 
In fact, for each $L\in U$ let $\pi_{L}$ be the spectral projector given  by 
$$
\pi_L=\int_{\partial D}(\lambda I-L)^{-1}d\lambda.
$$
Notice that the mapping  
$U\ni L\mapsto \pi_L\in L(X,X)$ is continuous, so 
by the Proposition (\ref{Espec.3}) if necessary one can shrink $U$ 
so that for all $L\in U$    
the application $\pi_L$ has the same rank as $\pi_T$. 
This fact together with the Remark \ref{Obs. Espec} 
implies that the portion of the spectrum of $L\in U$   
which lies in $\mathrm{int}(D)$ is not empty, call it  $\Sigma(L)$.
Define $X_{\Sigma(L)}\equiv \pi_{L}X$ and 
$T_{\Sigma(L)}=T|_{X_{\Sigma(L)}}$. 
It is well know \cite[p. 98]{EL62} that $\mathrm{spec}(T_{\Sigma(L)})=\Sigma(L)$.
If $\Sigma(L)$ is not an unitary set, then 
$X_{\Sigma(L)}$ is not a uni-dimensional subspace
and therefore  $1\neq \dim (X_{\Sigma(L)}) = \dim(\pi_{T}X)=1$.
This contradiction shows that there is a unique simple eigenvalue
$\lambda(L)$ of $L$ inside $D$.

By using the two previous claims one have a well defined
mapping $U\ni L\mapsto \lambda(L)\in \mathbb{C}$,  
where $\lambda(L)$ is the unique simple eigenvalue of $L$ inside 
$\mathrm{int}(D)$.
Now we proceed to the proof that this mapping is analytic. 
Fix $v\in X$ such that $\pi_L v$ is a non zero vector and choose 
$w\in X^{*}$ such that 
$w(\pi_{L}v )\equiv \left<w,\pi_{L}v\right>\neq 0$ (Hahn Banach)
for all $L$ in a small enough neighborhood of $T$.  
According  to the Proposition \ref{Espec.2} the operator 
$L$ commutes with $\pi_{L}$ then we get that 
$
\left<w,\pi_L(Lv)\right>
=
\left<w,L\pi_L(v)\right>
=
\lambda(L)\left<w,\pi_L(v)\right>
$
and consequently 
\begin{equation}\label{Espec. 7}
\lambda(L)=\dfrac{\left<w,\pi_L(Lv)\right>}{\left<w,\pi_L(v)\right>}.
\end{equation}
From the Definition \ref{Espec.1} and the above equality 
we obtain the analyticity of the mapping $U\ni L\mapsto \lambda(L)\in \mathbb{C}$.
\qed

\section{Spectral Gap and Exponential Decay}

In this section we will  follow closely \cite{Baladi}.

\begin{defnc}
Consider the probability space $(\Omega,\mathcal{F}, \nu)$.
Let $\sigma$ be the left shift on $\Omega$. 
For each $\varphi_1$ and $\varphi_2$ in $L^2(\Omega,\nu)$ 
we define the correlation function 
$C_{\varphi_1,\varphi_2,\nu}:\mathbb{Z}\to \mathbb{R}$ by
\begin{equation}
C_{\varphi_1,\varphi_2,\nu}(n)
=
\int_{\Omega} 
(\varphi_1\circ \sigma^n)\varphi_2\, d\nu
- 
\int_{\Omega}\varphi_1\, d\nu
\int_{\Omega}\varphi_2\, d\nu.
\end{equation}
\end{defnc}

\begin{teo}\label{D. Correlation}
Suppose that $f\in W(\Omega)$ 
is a potential for which the Ruelle operator 
$\mathscr{L}_{f}$ has the spectral gap property.
Consider the measure $\mu_{f}=h_{f}\nu_{f}$, where 
$\nu_{f}$ is the eigenmeasure given by the Theorem \ref{Ruelle-Walters}.
Then the correlation function $C_{\varphi_1,\varphi_2, \mu_{f}}(n)$ 
decays exponentially fast. 
More precisely, there are $0<\widetilde{\tau}<1$ and 
$C(\widetilde{\tau})>0$ such that 
for all $\varphi_1,\varphi_2\in  W(\Omega)$ 
the correlation function satisfies:
\begin{eqnarray}
|C_{\varphi_1,\varphi_2,\mu_{f}}(n)|
&=&
\left|
	\int_{\Omega} (\varphi_1\circ \sigma^n)\varphi_2\, d\mu_{f}
	- 
	\int_{\Omega}\varphi_1 d\mu_{f}\int_{\Omega}\varphi_2\, d\mu_{f}
\right|
\leq 
C_1\, \widetilde{\tau}^n.
\end{eqnarray}
where 
$
C_1
=
C(\widetilde{\tau}) 
\|h_{f}\|_{0} 
(\int_{\Omega}|\varphi_1|d\nu_{f})
\|\varphi_2\|
$.
\end{teo}

Before prove the above theorem, we present two auxiliary lemmas.
\begin{lema}\label{lema-proj-spectral-pif}
The spectral projection $\pi_{f}\equiv \pi_{\mathscr{L}_{f}} $ is given by 
$
\pi_{f}(\varphi) 
= 
\big(\int_{\Omega} \varphi\, d\nu_{f}\big) \cdot h_{f}.
$

\end{lema}

\begin{proof}
We know that $\pi_{f}$ and $\mathscr{L}_{f}$ commutes.  
By the Ruelle Theorem (Theorem \ref{Ruelle-Walters}) we have that 
$
\lim_{n\to \infty}(1/\lambda^n)
\mathscr{L}^n_{f}\varphi
= 
h_{f}
\int_{\Omega} \varphi\, d\nu_{f}
$
uniformly. Since $\pi_{f}$ is bounded we get that 
\begin{align*}
\left\|
	\pi_{f}(\lambda^{-n}\mathscr{L}^n_{f}
	\varphi-h_{f}
	\int_{\Omega} \varphi\, d\nu_{f})
\right\|_{0}
\leq 
\|\pi_{f}\|
	\left\|
		\lambda^{-n}\mathscr{L}^n_{f}\varphi
		-h_{f}
		\int_{\Omega} \varphi\, d\nu_{f}
	\right\|_{0}
\to 
0,
\end{align*}
when  $n\to \infty$. 
Since 
$
\pi_{f}(\lambda^{-n}\mathscr{L}^n_{f}\varphi))
=
\lambda^{-n}\mathscr{L}^n_{f}\pi_{f}(\varphi)
=
\lambda^n \lambda^{-n}\pi_{f}(\varphi)
=
\pi_{f}(\varphi)
$
we get that 
\begin{align*}
\pi_{f}(\varphi)
=
\pi_{f}\left( 
			h_{f}\int_{\Omega} \varphi\, d \nu_{f}
		\right)
=
\int_{\Omega} \varphi\, d \nu_{f} \cdot \pi_{f}(h_{f})
=
\int_{\Omega} \varphi\, d\nu_{f} \cdot h_{f}.
\end{align*}
\end{proof}

\begin{lema}\label{lema-phi1-fora}
Let be $\varphi_1,\varphi_2 \in W(\Omega)$ then
$
\mathscr{L}_{f}^n(\varphi_1\circ \sigma^n \cdot \varphi_2 \cdot h_{f})
=
\varphi_1 \mathscr{L}_{f}^n( \varphi_2 h_{f}). 
$
\end{lema}

\begin{proof}
The proof is an easy calculation. 
Let $x\in \Omega$ and $\varphi\in W(\Omega)$ then we have  that
$
\mathscr{L}_{f}^n\varphi(x) 
=
\int_{M^n}\varphi(\ba x)e^{S_n(\ba x)}d\mu(\ba).
$
Since 
$
\varphi_1\circ \sigma^n(\ba x)
=
\varphi_1(x)
~~\forall \ba x,~\ba\in M^n
$ 
and 
$\forall x\in \Omega$ we get that
\break
$
\mathscr{L}_{f}^n(\varphi_1\circ \sigma^n \varphi_2 h_{f})(x)
=
\int_{M^n}\varphi_1\circ \sigma^n \varphi_2 h_{f}(\ba x)e^{S_n(\ba x)}d\mu(\ba)
=
\varphi_1(x) \int_{M^n}(\varphi_2 h_{f})(\ba x)e^{S_n(\ba x)}d\mu(\ba).
$
\end{proof}

\noindent
{\bf Proof of the Theorem \ref{D. Correlation}}.
Since $\mu_{f}=h_{f} d\nu_{f}$ it follows from  
the definition of the correlation function that
\begin{align*}\label{D. Correlation 2}
|C_{\varphi_1,\varphi_2,\mu_{f} }(n)|
&=
\left| 
\int_{\Omega} (\varphi_1\circ \sigma^n)\varphi_2 h_{f}\, d\nu_{f}
- 
\int_{\Omega}\varphi_1  h_{f}\, d\nu_{f}
\int_{\Omega}\varphi_2 h_{f}\, d\nu_{f}                            
\right|.
\end{align*}
Notice that 
$
(\mathscr{L}^{*}_{f})^n\nu_{f}
=
\lambda_{f}^n\nu_{f}
$ 
and therefore the rhs above is equal to 
\[
\left|
\int_{\Omega}\lambda_{f}^{-n}
\mathscr{L}_{f}^n((\varphi_1\circ \sigma^n)\varphi_2 h_{f})
\, d\nu_{f}
-
\int_{\Omega}\varphi_1  h_{f}\, d\nu_{f}
\int_{\Omega}\varphi_2 h_{f}\,  d\nu_{f}                            
\right|
\]
By using the Lemma \ref{lema-phi1-fora} and performing simple 
algebraic computations we get 
\begin{align} \label{estimativa-1-dec-pol}
|C_{\varphi_1,\varphi_2,\mu_{f} }(n)|
\leq 
\left(   \int_{\Omega}|\varphi_1|\, d\nu_{f}    \right)
\left\|
\lambda^{-n}_{f} \mathscr{L}_{f}^n
\left(
	\varphi_2 h_{f}-  h_{f}\int_{\Omega} \varphi_2 h_{f}\, d\nu_{f}  
\right) 
\right\|_{0}.
\end{align}

We are supposing that the spectrum of 
$\mathscr{L}_{f}: W(\Omega)\to W(\Omega)$ 
consists in a simple eigenvalue $\lambda_{f}>0$ and a subset 
of a disc of radius strictly smaller than $\lambda_{f}.$  
Set 
$
\tau=\sup\{|z|; |z|<1~\textnormal{and}~z\cdot\lambda_{f} \in \sigma (\mathscr{L}_{f})\}.
$ 
The existence of the spectral gap guarantees that $\tau<1$. 
Let $\pi_{f}$ 
the spectral projection associated to eigenvalue $\lambda_{f}$,  
then by the Proposition \ref{Espec.5}, 
the spectral radius of the operator 
$\mathscr{L}_{f}(I-\pi_{f})$ is exactly $\tau\cdot \lambda_f$. 
Since the commutator $[\mathscr{L}_{f},\pi_{f}]=0$,
we get $\forall n \in\mathbb{N}$ that
$
[\mathscr{L}_{f}(I-\pi_{f})]^n
=
\mathscr{L}_{f}^{n}(I-\pi_{f}).
$
From the spectral radius formula \eqref{Spec. Radius} 
it follows that for each choice of $\widetilde{\tau}>\tau$ 
there is $n_0\equiv n_0(\widetilde{\tau})\in\mathbb{N}$ 
so that for all $n\geq n_0$ we have
$
\|\mathscr{L}_{f}^{n}(\varphi-\pi_{f}\varphi)\|
\leq 
\lambda_{f}^n\widetilde{\tau}^n
\|\varphi\|,~\forall\varphi\in  W(\Omega).
$
Therefore there is a constant 
$C( \widetilde{\tau})>0$ 
such that for every $n\geq 1$
\[ 
\|\mathscr{L}_{f}^n(\varphi-\pi_{f}\varphi)\|
\leq 
C( \widetilde{\tau}) 
\lambda_{f}^n\, \widetilde{\tau}^n\,
\|\varphi\|
\qquad\forall\varphi\in  W(\Omega).
\]
By using the Lemma \ref{lema-proj-spectral-pif} 
and the above upper bound 
in the inequality \eqref{estimativa-1-dec-pol}
we obtain 
\[
|C_{\varphi_1,\varphi_2,\mu_f}(n)|
\leq
\left( \int_{\Omega}|\varphi_1|\, d\nu_{f} \right)
C \widetilde{\tau}^n\|\varphi_2 h_{f}\|
\leq
C(\widetilde{\tau}) 
\|h_{f}\|_{0} 
\left( \int_{\Omega}|\varphi_1|d\nu_{f} \right)
\|\varphi_2\|.
\]
\qed

\section{Absence of the Spectral Gap in the Walters Space}

In this section we present the so-called long-range 
Ising model on the lattice $\mathbb{N}$ in the 
Thermodynamic Formalism setting. 
The goal is to exhibit explicitly a potential in the Walters space
for which the associated Ruelle operator do not have the spectral gap.

Throughout this section we assume 
the metric space $(M,d)$ 
is given by $(\{-1,1\},|\cdot|)$, where $|\cdot|$
is the modulus function and the a priori probability measure
$\nu = (1/2)[\delta_{\{-1\}}+\delta_{\{1\}}]$. 
Fix $\alpha>1$ and consider the potential
$f:\Omega\to\mathbb{R}$ given by
	\[
		f(x)
		= 
		-\sum_{n\geq 2} \frac{x_1x_{n}}{n^{\alpha}}.
	\]
This potential is not $\gamma$-H\"older
continuous for any $0<\gamma\leq 1$, 
see \cite{CL14}. 
When $1<\alpha<2$, Dyson \cite{dyson} proved that this model
has spontaneous magnetization for sufficiently low temperatures. 
This fact for these models implies non-uniqueness DLR-Gibbs 
measures at such temperatures and also that the pressure can not
be Fr\'echet-differentiable on a suitable Banach space. For $\alpha=2$
this phase transition result was proved 
by Fr\"olich and Spencer \cite{frolich-spencer}.
On the other hand, when $\alpha>2$ 
the potential $f$ is in the Walters class.
Indeed, for any choice of $n,p\in\mathbb{N}$ we have
$
\mathrm{var}_{n+p} (\, f(x) + f(\sigma(x))+...+f(\sigma^{n-1}(x))\,)
=
(n+p)^{-\alpha+1} + (n+p-1)^{-\alpha+1}+...+ p^{-\alpha+1}
$
which implies that 
$$
\sup_{n\in \mathbb{N}}\left[
\mathrm{var}_{n+p} (\, f(x) + f(\sigma(x))+...+f(\sigma^{n-1}(x))\,)\right]
\sim \sum_{j=p}^\infty j^{-\alpha+1}\sim p^{-\alpha +2},$$
and then the Walters condition. For this reason, in what 
follows we assume that $\alpha>2$.
In this case as mentioned in the introduction the potential $f$ 
belongs to an infinite dimensional subspace of $C(\Omega)$ 
as defined in \cite{CO81} where the pressure is Fréchet-analytic.
Note that the previous computation implies that this
space can not be contained in the H\"older space.

In the Statistical Mechanics setting 
the potential $f$ is normally replaced/constructed by the
absolutely uniformly summable interaction  
$\Phi = (\Phi_{A})_{A\Subset \mathbb{N}}$, given by 
	\[
		\Phi_{A}(x) =
		\begin{cases}
			\displaystyle\frac{x_{n}x_{m}}{|n-m|^{\alpha}},
					&\ \text{if}\ \ A=\{n,m\}\subset\mathbb{N}
					\ \text{and}\ m\neq n;
			\\
			0,			&\ \text{otherwise}.
		\end{cases}
	\]
The relationship between the potential $f$ and 
the interaction $\Phi$ is detailed described in 
\cite{CL14} and expressed by the following equality
\begin{equation}
		H_{n}(x)
		=
		\sum_{\substack{ A\Subset \mathbb{N}
						\\
						A\cap \Lambda_n\neq \emptyset} }
		\Phi_A(x)
		=
		f(x)+f(\sigma x)+\ldots+f(\sigma^{n-1} x).
\end{equation}

Following \cite{CL14,Geogii88,Ruelle} 
we construct a DLR-Gibbs measure as follows.
Fix a potential $f$, and a {\it boundary condition},
which here for convenience will be chosen as $y=(1,1,\ldots)\in\Omega$. 
Now we take any cluster point
(with respect to the weak topology)  
of the sequence $(\nu^{y}_n)_{n\in\mathbb{N}}$ in $\mathscr{P}(\Omega,\mathcal{F})$,
where $\nu^{y}_n:\mathcal{F}\to [0,1]$ is the probability measure  
defined for each $F\in \mathcal{F}$ by the following expression:  
	\[
		\nu_{n}^{y}(F) = \frac{1}{Z_{n}^{y}}
		\sum_{ \substack{ x\in\Omega;\\ \sigma^n(x)=\sigma^n(y) }   }
		\!\!\!\!\!\!\!\!
		1_{F}(x)\exp(H_n(x)),
\quad\text{where}\quad
		Z_{n}^{y}
		=
		\sum_{ \substack{ x\in\Omega;\\ \sigma^n(x)=\sigma^n(y) }   }
		\!\!\!\!\!\!\!\!
		\exp(H_n(x)).
	\]

Since $\alpha>2$ it is well-known that the sequence $(\nu^{y}_n)_{n\in\mathbb{N}}$
has a unique cluster point which will be denoted by $\nu_{\mathbb{N}}$.
A proof of this classical fact, with a dynamical system point of view,
using a consequence of the 
Dobrushin uniqueness theorem and also the Ruelle operator 
formalism is presented in \cite{CL14}.

Our next step is to construct a probability measure 
$\nu_{\mathbb{Z}}$ on the symbolic space 
$
\hat{\Omega}
\equiv 
\{-1,1\}^{\mathbb{Z}}
\equiv 
\{-1,1\}^{\mathbb{Z}\cap (-\infty,0]}
\times
\{-1,1\}^{\mathbb{N}}
$
such that 
\begin{equation}
\label{eq-medz-medn}
\nu_{\mathbb{N}}(F)
=
\nu_{\mathbb{Z}}(\{-1,1\}^{\mathbb{Z}\cap (-\infty,0]}\times F),
\quad
\forall F\in\mathcal{F}.
\end{equation}

Let us denote 
$
\mathrm{Diag}(\mathbb{Z}\times\mathbb{Z})
\equiv
\{
(r,r): r\in\mathbb{Z}
\}
$
and 
$
\mathbb{M}
\equiv
\mathbb{Z}\times\mathbb{Z}
\setminus
\mathrm{Diag}(\mathbb{Z}\times\mathbb{Z}).
$
We define a linear space 
$
\mathbb{J}
\subset 
\mathbb{R}^{\mathbb{M}}
\equiv
\{
J_{ij}\in\mathbb{R}: (i,j)\in\mathbb{M}
\}
$
as being the set of points in $\mathbb{R}^{\mathbb{M}}$ satisfying 
$\sup_{i\in\mathbb{Z}}\sum_{j\in\mathbb{Z}:j\neq i} |J_{ij}|<\infty$.
Let $J^{\Phi}$ and $T^{\Phi}$ be two points in $\mathbb{J}$
defined by  
$
(J^{\Phi})_{ij} = |i-j|^{-\alpha}
$
for all $(i,j)\in\mathbb{M}$ 
and $(T^{\Phi})_{ij}= (J^{\Phi})_{ij}$
if $i,j\in\mathbb{N}$ with $i\neq j$ 
and $(T^{\Phi})_{ij}=0$ otherwise.
For each $n\in\mathbb{N}$ and $J\in \mathbb{J}$ 
we define the function 
$\mathscr{H}_n:\hat{\Omega}\times\mathbb{J}\to\mathbb{R}$
by 
\begin{equation}\label{hami-interpol}
\mathscr{H}_n(z,J)
=
\sum_{i=-n}^n\ \sum_{j\in\mathbb{Z}: j\neq i} J_{ij}\, z_iz_j.
\end{equation}
For any fixed $J\in\mathbb{J}$ and $\hat{y}=(\ldots,1,1,1,\ldots)\in\hat{\Omega}$
we can define, similarly as above, a probability measure 
$\nu_{n}^{\hat{y},J}$ such that for each borelian $\hat{F}$
of $\hat{\Omega}$ we have 
\begin{equation}
\label{med-vol-fin-Z}
		\nu_{n}^{\hat{y},J}(\hat{F}) 
		= 
		\frac{1}{Z_{n}^{\hat{y},J}}
		\!\!\!\!\!
		\sum_{
			\substack{
				z\in\hat{\Omega};\, z_i=1
				\\ 
				\forall\, i\in\mathbb{Z}\setminus\{-n,\ldots,n\} 
			}   
		}
		\!\!\!\!\!\!\!\!
		1_{\hat{F}}(z)\exp(\mathscr{H}_n(z,J)),
\quad\text{where}\quad
		Z_{n}^{\hat{y},J}
		=
		\!\!\!\!\!
		\sum_{
			\substack{
				z\in\hat{\Omega};\, z_i=1
				\\ 
				\forall\, i\in\mathbb{Z}\setminus\{-n,\ldots,n\} 
			}   
		}
		\!\!\!\!\!\!\!\!
		\exp(\mathscr{H}_n(z,J)).
\end{equation}
By straightforward computation we obtain, 
for each $n\in\mathbb{N}$ and $z\in\hat{\Omega}$ fixed,
the following identities:
\begin{enumerate}
\item
$
\mathscr{H}_n(z,T^{\Phi})
=
H_n(z_1,z_2,\ldots);
$
\item
$
Z_{n}^{\hat{y},T^{\Phi}}
=
2^{n+1} Z_{n}^{y};
$
\item
$
1_{ \{\{-1,1\}^{\mathbb{Z}\cap (-\infty,0]}\times F\} }(z)
=
1_{F}(z_1,z_2,\ldots).
$
\end{enumerate}
Using the above three identities one can immediately see that 
\begin{align}\label{id-aux-nuz-nun-sec7}
\nu_{n}^{\hat{y},T^{\Phi} }(\{\{-1,1\}^{\mathbb{Z}\cap (-\infty,0]}\times F) 
&= 
\frac{1}{2^{n+1} Z_{n}^{y}}
\!\!\!\!\!
\sum_{
	\substack{
		z\in\hat{\Omega};\, z_i=1
		\\ 
		\forall\, i\in\mathbb{Z}\setminus\{-n,\ldots,n\} 
	}   
}
\!\!\!\!\!\!\!\!
1_{F}(z_1,z_2,\ldots)\exp(- H_n(z_1,z_2,\ldots)
\nonumber
\\
&=
\nu_{n}^{y}(F).
\end{align} 
Recalling that $\alpha>2$, it follows from classical
results of the theory of DLR-Gibbs measures that   
the sequence $(\nu_{n}^{\hat{y},J^{\Phi} })_{n\in\mathbb{N}}$ 
has a unique cluster point which we call $\nu_{\mathbb{Z}}$.
From the previous equality it easy to conclude that 
\eqref{eq-medz-medn} is valid.

According to the definition of $\nu_{n}^{\hat{y},J}$, 
for any fixed measurable set $\hat{F}$ 
and $n\in\mathbb{N}$  the function
$
\mathbb{J}
\ni 
J
\mapsto 
\nu_{n}^{\hat{y},J}(\hat{F}) 
$
is Fr\'echet-analytic, since it is just a finite sum 
of analytic functions. A straightforward computation 
shows that for each fixed $(i,j)\in\mathbb{M}$
we have
\[
\frac{\partial}{\partial J_{ij}}
		Z_{n}^{\hat{y},J}
		=
		\!\!\!\!\!
		\sum_{
			\substack{
				z\in\hat{\Omega};\, z_i=1
				\\ 
				\forall\, i\in\mathbb{Z}\setminus\{-n,\ldots,n\} 
			}   
		}
		\!\!\!\!\!\!\!\!
		\frac{\partial}{\partial J_{ij}}
		(\mathscr{H}_n(z,J))
		\cdot
		\exp(\mathscr{H}_n(z,J))
=
		\!\!\!\!\!
		\sum_{
			\substack{
				z\in\hat{\Omega};\, z_i=1
				\\ 
				\forall\, i\in\mathbb{Z}\setminus\{-n,\ldots,n\} 
			}   
		}
		\!\!\!\!\!\!\!\!
		z_iz_j
		\cdot
		\exp(\mathscr{H}_n(z,J)).
\]
By multiplying and dividing the rhs by $Z_{n}^{\hat{y},J}$ 
and use the definition of the Lebesgue integral 
we get
\[
\frac{\partial}{\partial J_{ij}}
		Z_{n}^{\hat{y},J}
		=
		Z_{n}^{\hat{y},J}
		\int_{\hat{\Omega}}
		z_iz_j\,
		d\nu_{n}^{y,J}(z).
\]
Performing similar computations and using the quotient rule, 
we have for any measurable function $\varphi:\hat{\Omega}\to\mathbb{R}$ 
\begin{align}\label{derivada-prob-sec7}
\frac{\partial}{\partial J_{ij}}
		\nu_{n}^{\hat{y},J}(\varphi) 
&= 
\frac{\partial}{\partial J_{ij}}
\big[
		\frac{1}{Z_{n}^{\hat{y},J}}
		\!\!\!\!\!
		\sum_{
			\substack{
				z\in\hat{\Omega};\, z_i=1
				\\ 
				\forall\, i\in\mathbb{Z}\setminus\{-n,\ldots,n\} 
			}   
		}
		\!\!\!\!\!\!\!\!
		\varphi(z)\exp(\mathscr{H}_n(z,J))
\big]
\nonumber
\\
&=		
\int_{\hat{\Omega}} \varphi(z)\, z_iz_j\, d\nu_{n}^{y,J}(z)
-
\int_{\hat{\Omega}} \varphi(z)\, d\nu_{n}^{y,J}(z)
\int_{\hat{\Omega}} z_iz_j\, d\nu_{n}^{y,J}(z).
\end{align}

Before proceed we state the GKS-II inequality but only in the generality
required in this section. For more general cases, see 
\cite{Ellis,griffiths2,KS68,ginibre70}.  
\begin{teo}[GKS-II Inequality \cite{griffiths2,KS68}]
Fix $n\in\mathbb{N}$ and $\{n_1,n_2,\ldots,n_k\}$ 
an arbitrary subset of $\{-n,\ldots,n\}$.
If $J\in \mathbb{J}$ satisfies $J_{ij}\geqslant 0$ for all $(i,j)\in\mathbb{M}$,
then 
\[
\int_{\hat{\Omega}} 
z_{n_1}\cdot\ldots\cdot z_{n_k}
\cdot 
z_i\cdot z_j\,\, 
d\nu_{n}^{y,J}(z)
-
\int_{\hat{\Omega}} z_{n_1}\cdot\ldots\cdot z_{n_k}\,\, d\nu_{n}^{y,J}(z)
\int_{\hat{\Omega}} z_i\cdot z_j\,\, d\nu_{n}^{y,J}(z)
\geqslant
0,
\]
where $\nu_{n}^{y,J}$ denotes the probability measure defined in 
\eqref{med-vol-fin-Z}.
\end{teo}

From now on we take $\varphi(z)=z_1$. Strictly speaking $\varphi$
is defined on $\hat{\Omega}$ but we will abuse notation and 
also use $\varphi(x)$ to denote the projection on the first coordinate
of an element in $\Omega$. 
If $\widetilde{J}\in\mathbb{J}$ is such that 
$\widetilde{J}_{ij}\geqslant 0$ for all $(i,j)\in\mathbb{M}$
it follows from \eqref{derivada-prob-sec7} 
and GKS-II inequality that 
\[
\left.
\frac{\partial}{\partial J_{ij}}
		\nu_{n}^{\hat{y},J}(\varphi)
\right|_{J=\widetilde{J}}		
\geqslant
0.
\]
This inequality implies that the mapping 
$\mathbb{J}\ni J\mapsto \nu_{n}^{\hat{y},J}(\varphi)$
is coordinatewise non-decreasing in 
$\mathbb{J}\cap [0,+\infty)^{\mathbb{M}}$. 
This monotonicity together with the inequalities 
$(T^{\Phi})_{ij}\leqslant (J^{\Phi})_{ij}$ 
immediately implies that 
\[
		\int_{\hat{\Omega}}
		\varphi(z)\,
		d\nu_{n}^{y,T^{\Phi}}(z)
		\leqslant
		\int_{\hat{\Omega}}
		\varphi(z)\,
		d\nu_{n}^{y,J^{\Phi}}(z).
\]
Since $\varphi$ is a simple function taking only the values $-1$ and $1$
the lhs above is, from definition of the Lebesgue integral 
and the identity \eqref{id-aux-nuz-nun-sec7}, equals to  
\begin{align*}
		\int_{\hat{\Omega}}
		\varphi(z)\,
		d\nu_{n}^{y,T^{\Phi}}(z)
&=
\nu_{n}^{y,T^{\Phi}}(\{z\in\hat{\Omega}:z_1=1\})
-
\nu_{n}^{y,T^{\Phi}}(z\in\hat{\Omega}:z_1=-1)
\\
&=
\nu_{n}^{y}(\{x\in\Omega:x_1=1\})
-
\nu_{n}^{y}(x\in\Omega:x_1=-1)
\\
&=
		\int_{\Omega}
		\varphi(x)\,
		d\nu_{n}^{y}(x).
\end{align*}
Replacing this last expression in the above inequality we arrive at
\begin{equation}\label{des-aux-mag-zero}
		\int_{\Omega}
		\varphi(x)\,
		d\nu_{n}^{y}(x)
		\leqslant
		\int_{\hat{\Omega}}
		\varphi(z)\,
		d\nu_{n}^{y,J^{\Phi}}(z).
\end{equation}
Since $\nu_{n}^{y,J^{\Phi}}\rightharpoonup \nu_{\mathbb{Z}}$
and $\nu_{n}^{y}\rightharpoonup \nu_{\mathbb{N}}$
it follows from the definition of weak convergence that 
\[
		\int_{\Omega}
		\varphi(x)\,
		d\nu_{\mathbb{N}}(x)
		\leqslant
		\int_{\hat{\Omega}}
		\varphi(z)\,
		d\nu_{\mathbb{Z}}(z).
\]
To prove that the lhs above is non-negative
we will use the GKS-I inequality. Here we also
state it in the needed particular case. 
Its proof as well as its more general version can be found 
in \cite{Ellis,griffiths1,ginibre70}.
\begin{teo}[GKS-I Inequality \cite{griffiths1}]
Fix a natural number $n\geq 1$ and subset $\{n_1,n_2,\ldots,n_k\}\subset \{-n,\ldots,n\}$
and $J\in \mathbb{J}$ satisfying $J_{ij}\geqslant 0$ for all $(i,j)\in\mathbb{M}$. 
If $\nu_{n}^{y,J}$ denotes the probability measure defined in 
\eqref{med-vol-fin-Z}, then 
\[
\int_{\hat{\Omega}} 
z_{n_1}\cdot\ldots\cdot z_{n_k}
\cdot 
z_i\cdot z_j\,\, 
d\nu_{n}^{y,J}(z)
\geqslant
0.
\] 
\end{teo}

By applying GKS-I inequality to the lhs of \eqref{des-aux-mag-zero} 
and then taking the weak limit, when $n\to\infty$  
we get 
\[
0
\leqslant
		\int_{\Omega}
		\varphi(x)\,
		d\nu_{\mathbb{N}}(x)
		\leqslant
		\int_{\hat{\Omega}}
		\varphi(z)\,
		d\nu_{\mathbb{Z}}(z).
\]
Since $\alpha>2$,  there is a theorem ensuring that 
$\int_{\hat{\Omega}} \varphi\, d\nu_{\mathbb{Z}}(z)=0$,
see \cite{Ellis}. Therefore we have proved that
\begin{equation}\label{mag-zero}
\int_{\Omega}
\varphi(x)\,
d\nu_{\mathbb{N}}(x)
=
0.
\end{equation}

To get the lower bound we are interested  
we define for $n\geq 1$ the element
$J^{\{1,n+1\}}\in\mathbb{J}$, where 
$(J^{\{1,n+1\}})_{ij}=n^{-\alpha}$
if $(i,j)=(1,n+1)$ and $(J^{\{1,n+1\}})_{ij}=0$
otherwise. 
Similarly as above, we obtain by another application 
of the GKS-II inequality the coordinatewise monotonicity 
of the mapping 
$
\mathbb{J}\ni J
\mapsto 
\nu_{m}^{\hat{y},J}(z_{n+1}z_1),
$
where $m>n$, 
therefore we can conclude that  
\[
		\int_{\hat{\Omega}}
		z_{n+1}z_1\,
		d\nu_{m}^{y,J^{\{1,n+1\}}}(z)
		\leqslant
		\int_{\hat{\Omega}}
		z_{n+1}z_1\,
		d\nu_{m}^{y,J^{\Phi}}(z).
\]
Notice that the lhs above can be explicitly computed as follows
(and its value is independent of $m$)
\begin{align*}
\int_{\hat{\Omega}}
z_{n+1}z_1\,
d\nu_{m}^{y,J^{\{1,n+1\}}}(z)
&=
\sum_{z_{n+1}=\pm 1}\sum_{z_1=\pm 1}
z_{n+1}z_1
\exp(\frac{z_{n+1}z_1}{n^{\alpha}})
\left[
\sum_{z_{n+1}=\pm 1}\sum_{z_1=\pm 1}
\exp(\frac{z_{n+1}z_1}{n^{\alpha}})
\right]^{-1}
\\[0.3cm]
&=
\frac{2\exp(n^{-\alpha})-2\exp(n^{-\alpha})}
{2\exp(n^{-\alpha})+2\exp(n^{-\alpha})}
\\[0.3cm]
&=
\tanh(n^{-\alpha}).
\end{align*}  
On the other hand, by using the previous equality 
and \eqref{id-aux-nuz-nun-sec7} we get for any $m>n$ 
\[
\tanh(n^{-\alpha})
\leqslant
\int_{\hat{\Omega}}
z_{n+1}z_1\,
d\nu_{m}^{y,J^{\Phi}}(z)
=
\int_{\Omega}
x_{n+1}x_1\,
d\nu_{m}^{y}(x)
=
\int_{\Omega}
(\varphi\circ \sigma^n)\varphi\,
d\nu_{m}^{y}(x).
\]
By Taylor expanding the hyperbolic tangent and 
taking the weak limit when $m\to\infty$, we get 
the following inequality for some constant $C>0$ 
\[
\int_{\Omega} 
(\varphi\circ \sigma^n)\varphi\,\, d\nu_{\mathbb{N}}
\geqslant
\tanh(n^{-\alpha})
\geqslant
\frac{C}{|n|^{\alpha}}.
\]
Piecing together the previous inequality and \eqref{mag-zero}
we finally arrived at
\begin{equation}\label{dec-pol-ising}
\frac{C}{|n|^{\alpha}}
\leqslant
\int_{\Omega} 
(\varphi\circ \sigma^n)\varphi\, d\nu_{\mathbb{N}}
- 
\int_{\Omega}\varphi\, d\nu_{\mathbb{N}}
\int_{\Omega}\varphi\, d\nu_{\mathbb{N}}
=
C_{\varphi,\varphi,\nu_{\mathbb{N}}}(n).
\end{equation}

It was shown in \cite{CL14} 
that 
$\mu_{f}\equiv h_fd\nu_{f}$, where $\nu_{f}$ and $h_f$ 
is given by Theorem \ref{Ruelle-Walters}, belongs 
to $\mathcal{G}^{DLR}(f)$.
The authors also shown that for $\alpha>2$ 
the set $\mathcal{G}^{DLR}(f)$ is a singleton 
and therefore $\mu_{f}=\nu_{\mathbb{N}}$. This fact 
together with the continuity of $h_f$ and 
the previous inequality shows that 
$C_{\varphi,\varphi,\mu_f}(n)$ can not decays
exponentially fast. Since $\varphi(x)=x_1$ is in the
Walters class it follows from the 
Theorem \ref{D. Correlation} that $\mathscr{L}_{f}$
has not the spectral gap property.

\section{Appendix}

\subsubsection*{Analyticity on Banach Spaces}

\begin{defnc}
\label{der-ordemsup}
Let $(X, \|\cdot \|_X)$ and $(Y, \|\cdot \|_Y)$ 
be Banach spaces and $U$ an open subset of $X$. 
For each $k\in \mathbb{N}$, 
a function $F:U\rightarrow Y$ is said to be 
$k$-differentiable in $x$ if for $j \in \{1,...,k\}$, 
there exist a $j$-linear bounded transformation 
$D^j F(x): X^j \rightarrow Y$ 
such that
$$
D^{j-1}F(x+v_j)(v_1,...,v_{j-1}) 
- 
D^{j-1}F(x)(v_1,...,v_{j-1}) 
= 
D^j F(x) (v_1,...,v_j) + o_j(v_j)
$$
where $o_j:X \rightarrow Y$ is such that 
$\lim_{v\rar 0} \|o_j(v)\|_Y/(\| v \|_X) = 0$.
\end{defnc}

We say that $F$ has derivatives of all orders in $U$, 
if for any $k\in \mathbb{N}$, and any $x \in U$, 
$F$ is $k$-differentiable in $x$.

\begin{defnc}
\label{analiticidade}
Let $X$ and $Y$ be Banach spaces and $U$ an open subset of $X$. 
A function $F:U\rightarrow Y$ is called analytic on $U$, 
when $F$ has derivatives of all orders in $U$, 
and for each $x\in U$ there exists an open neighborhood 
$U_x$ of $x$ in $U$ such that for all $v\in U_x$, we have that
$$
F(x+v)-F(x)=\sum \limits_{j=1}^{\infty}\frac{1}{n!}D^jF(x) v^j\,,
$$
where $D^j F(x) v^j=D^jF(x)(v,\ldots ,v)$ 
and $D^j F(x)$ is the $j$-th derivative of $F$  in $x$.
\end{defnc}

\noindent
If $F:U\rightarrow Y$ is analytic on $U$, 
then for each $n\in \mathbb{N}$, 
the Taylor expansion of order $n$ is
\begin{equation}\label{Taylor}
    F(x+v)
    =
    F(x)+D^1F(x)v
    +
    \frac{D^2F(x)v^2}{2}+\frac{D^3F(x)v^3}{6}+
    \hdots 
    + 
    \frac{D^nF(x)v^n}{n!}+o_{n+1}(v)\,,
\end{equation}
where
$
o_{n+1}(v)
=
\sum_{j=n+1}^{\infty}(1/n!)D^jF(x) v^j
$ 
satisfies 
$\lim_{v \rar 0} \|o_{n+1}(v)\|_Y/ \|v\|_X^n=0$.

\subsubsection*{Some background on spectral theory}
In this section we list some classical results of Spectral Theory
for more details and proofs see \cite{EL62}. 
Let $X$ be  a Banach space and  $T:X\to X$ a bounded operator, 
we define the spectrum  of the operator $T$ by
$$
\mathrm{spec}(T)=
\{\lambda\in \mathbb{C};(\lambda I-T)^{-1}~\textnormal{do not  exists}\}.
$$
The resolvent set $\rho(T)$ of $T$ is defined as the 
complement of $\mathrm{spec}(T)$. 
The resolvent set of a bounded operator is an open set 
while the spectrum is a compact set.
The spectral radius of the operator $T$ is defined as  
$r(T)=\sup\{|x-y|;~ x,y\in \mathrm{spec}(T)\}$. 
The spectral radius has the following characterization 
\begin{equation}\label{Spec. Radius}
r(T)
=
\liminf_{n}\|T^n\|^{\frac{1}{n}}
=
\lim_{n\to \infty}\|T^n\|^{\frac{1}{n}}.
\end{equation}
It is also known that $\mathrm{spec}(T)\subset B(0,r(T))$
and $\mathrm{spec}(T)=\mathrm{spec}(T^{*})$, 
where $T^*:X^{*}\to X^{*}$ is the adjoint of $T.$

\begin{defnc}\label{Espec.1}
Let   $T:X\to X$ be a bounded linear operator 
and $\gamma$ a  rectifiable Jordan curve that lies 
in $\rho(T)$, then we define the spectral projection 
$\pi_T:X\to X$ as follows 
$$
\pi_T
= 
\dfrac{1}{2\pi}\int_{\gamma}(\lambda I-T)^{-1}d\lambda \,.
$$
\end{defnc}

\begin{obs} \label{Obs. Espec}
If the interior of $\gamma$ lies in  the 
interior of $\rho(T)$ then $\pi_{T}=0$. 
On the other hand if $\mathrm{spec}(T)$ lies entirely 
in the interior of $\gamma$ then $\pi_L=Id$. 
\end{obs}

\begin{prop}\label{Espec.2}
If   $T:X\to X$ is bounded then $\pi_T$ 
is a  projection, i.e, $\pi_T^2=\pi_T$. 
Moreover  $\pi_T$ commutes with $T$
\end{prop}

A subset of $\mathrm{spec}(T)$ which is both open and closed  
in $\mathrm{spec}(T)$ is called a spectral set. 
Let $\Sigma(T)\subset \mathrm{spec}(T)$ be  a spectral set, 
and $\gamma$ a rectifiable  
Jordan curve which lies in $\rho(T)$  
containing $\Sigma(T)$ in its interior. 
Denote by $\pi_{T,\Sigma(T)}$ the spectral 
projection associated with $T$ and $\gamma$,
i.e., 
\[
	\pi_{T,\Sigma(T)}
	=
 	\dfrac{1}{2\pi}\int_{\gamma}(\lambda I-T)^{-1}d\lambda,
\]
where $\gamma$ is any rectifiable Jordan curve surrounding the 
spectral set $\Sigma(T)$, completely contained in the $\rho(T)$ 
and such that any other point in the spectrum is outside $\gamma$.

We use the notation $X_{\Sigma(T)}=\pi_{T,\Sigma(T)}X$ and 
$T_{\Sigma(T)}=T|_{X_{\Sigma(T)}}$. 

\begin{prop}\label{Espec.5}
Let  be $\Sigma(T)$ a spectral set of $spec(T)$ 
then $\mathrm{spec}(T_{\Sigma(T)})=\Sigma(T) $.
\end{prop}

\begin{prop}\label{Espec.3}
Let $\pi_1,\pi_2:X\to X$ be linear projections, 
then there exists $\epsilon>0$ such that if $\|\pi_1-\pi_2\|<\epsilon$ then 
$\pi_1$ and $\pi_2$ has the same rank, i.e,  $\dim\pi_1(X)=
\dim\pi_2(X)$.

\end{prop}

\section*{Acknowledgments}
We thank Aernout van Enter 
for calling to our attention the references 
\cite{CO81,dobrushin73}. We also thank Artur Lopes and 
Rodrigo Bissacot for their suggestions and comments 
on the early version of this manuscript.
Leandro Cioletti is partially supported by FEMAT and 
Eduardo Silva is supported by CNPq.

\begin{flushleft}

\bigskip

{\sc 
Leandro Cioletti\\
Eduardo Ant\^onio da Silva
\\[0.2cm]
Departamento de Matem\'atica\\
Universidade de Bras\'ilia\\
Campus Universit\'ario Darcy Ribeiro - Asa Norte\\
70910-900  Bras\'ilia - DF - Brazil.}
\\
{\it cioletti@mat.unb.br}\\
{\it eduardo23maf@gmail.com}
\end{flushleft}

\end{document}